\documentclass[a4paper,11pt,twoside,reqno]{article}

\usepackage{lipsum}
\usepackage{amsfonts}
\usepackage{graphicx}
\usepackage{epstopdf}
\usepackage{algorithmic}
\ifpdf
 \DeclareGraphicsExtensions{.eps,.pdf,.png,.jpg}
\else
 \DeclareGraphicsExtensions{.eps}
\fi

\usepackage{amsmath, amsfonts, amssymb, amsthm, mathtools}
\usepackage{stmaryrd}
\usepackage{bbm}
\usepackage{cases}
\usepackage{hyperref}
\usepackage[T1]{fontenc}
\usepackage[english]{babel}
\usepackage[cyr]{aeguill}
\usepackage{xspace}
\usepackage{graphicx}
\usepackage{caption}
\usepackage{subcaption}
\usepackage{color}
\usepackage{epigraph}
\usepackage{etoolbox}
\usepackage{enumitem}
\usepackage{stmaryrd}
\usepackage[margin=1in]{geometry}
\usepackage{cite}
\usepackage{dsfont}
\usepackage{ushort}
\usepackage{xcolor}
\usepackage{tikz}
\usetikzlibrary{arrows,backgrounds}
\usepgflibrary{shapes.multipart}
\usetikzlibrary{arrows.meta}
\usepackage{cleveref}
\usepackage{amsopn}

\usepackage{accents}
\newcommand{\ubar}[1]{\underaccent{\bar}{#1}}

\def\dbE{\mathbb{E}}
\def\dbF{\mathbb{F}}

\def\dbN{\mathbb{N}}
\def\dbP{\mathbb{P}}
\def\dbR{\mathbb{R}}

\def\dbQ{\mathbb{Q}}

\def\cA{{\cal A}}

\def\cC{{\cal C}}
\def\cD{{\cal D}}

\def\cF{{\cal F}}
\def\fF{{\cal F}}

\def\cK{{\cal K}}
\def\cL{{\cal L}}

\def\cP{{\cal P}}

\def\cT{{\cal T}}

\def\cW{{\cal W}}

\def\fF{{\mathfrak{F}}}

\def\a{\alpha}

\def\d{\delta}
\def\e{\varepsilon}

\def\l{\lambda}

\def\si{\sigma}
\def\t{\tau}
\def\f{\varphi}

\def\cd{\cdot}

\def\rc{{\rm rc}}
\def\syc{{\rm sc}}

\def\ol{\bar}
\def\ul{\ubar}

\def\ie{\textit{i.e.}}

\def\q{\quad}

\DeclareMathOperator*{\argmin}{\arg\!\min}

\DeclareMathOperator{\Expect}{\mathbb E}

\newcommand{\be}{\begin{equation}}
\newcommand{\ee}{\end{equation}}
\newcommand{\bee}{\begin{equation*}}
\newcommand{\eee}{\end{equation*}}
\newcommand{\bea}{\begin{eqnarray}}
\newcommand{\eea}{\end{eqnarray}}
\newcommand{\beaa}{\begin{eqnarray*}}
\newcommand{\eeaa}{\end{eqnarray*}}

\newtheorem{thm}{Theorem}[section]
\newtheorem{lem}[thm]{Lemma}
\newtheorem{cor}[thm]{Corollary}
\newtheorem{prop}[thm]{Proposition}
\newtheorem{rem}[thm]{Remark}

\newtheorem{defn}[thm]{Definition}
\newtheorem{assum}[thm]{Assumption}

\crefname{thm}{theorem}{theorems}
\begin{document}

% Title. If the supplement option is on, then "Supplementary Material"
% is automatically inserted before the title.
\title{Mean Field Optimization Problem Regularized by Fisher Information}

% Authors: full names plus addresses.
\author{Julien CLAISSE \footnote{Ceremade, Universit\'e Paris-Dauphine, PSL Research University, 75016 Paris, France (\it{claisse@ceremade.dauphine.fr}).}
\and Giovanni CONFORTI \footnote{Centre de Math\'ematiques Appliqu\'ees, Ecole Polytechnique, IP Paris, 91128 Palaiseau Cedex, France (\it{giovanni.conforti@polytechnique.edu}).}
\and Zhenjie REN \footnote{Ceremade, Universit\'e Paris-Dauphine, PSL Research University, 75016 Paris, France (\it{ren@ceremade.dauphine.fr}).}
\and Songbo WANG \footnote{Centre de Math\'ematiques Appliqu\'ees, Ecole Polytechnique, IP Paris, 91128 Palaiseau Cedex, France (\it{songbo.wang@polytechnique.edu}).}
}

%%% Local Variables:
%%% mode:latex
%%% TeX-master: "ex_article"
%%% End:

% FundRef data to be entered by SIAM
%<funding-group specific-use="FundRef">
%<award-group>
%<funding-source>
%<named-content content-type="funder-name">
%</named-content>
%<named-content content-type="funder-identifier">
%</named-content>
%</funding-source>\bibliography{references}
%<award-id> </award-id>
%</award-group>
%</funding-group>

\maketitle

\renewcommand {\theequation}{\arabic{section}.\arabic{equation}}
\def\thesection{\arabic{section}}

\numberwithin{equation}{section}
\numberwithin{thm}{section}

\begin{abstract}
Recently there is a rising interest in the research of mean field optimization, in particular because of its role in analyzing the training of neural networks. In this paper by adding the Fisher Information as the regularizer, we relate the  regularized mean field optimization problem to a so-called mean field Schr\"odinger (MFS for short) dynamics. We develop an energy-dissipation method to show that the marginal distributions of the MFS dynamics converge exponentially quickly towards the unique minimizer of the regularized optimization problem. Remarkably, the MFS dynamics is  proved to be a gradient flow on the probability measure space with respect to the relative entropy. Finally we propose a Monte Carlo method to sample the marginal distributions of the MFS dynamics. 
\end{abstract}

\section{Introduction}

Recently the mean field optimization problem, namely 
\[\inf_{p\in \cP} \fF(p), \quad\mbox{for a function $\fF:\cP\rightarrow\dbR$, where $\cP$ is a set of probability measures},\]
attracts increasing attention, in particular because of its role in analysing the training of artificial neural networks. The Universal Representation Theorem (see e.g. \cite{Hor91}) ensures that a given continous function $f:\dbR^d\rightarrow \dbR$ can be approximated by the parametric form:
\[f(x)  \approx \sum_{i=1}^N c_i \f(a_i\cdot x + b_i),
\quad\mbox{with $c_i\in \dbR, ~ a_i\in \dbR^d, ~b_i\in \dbR$ for $1\le i\le N$},\]
where $\f$ is a fixed non-constant, bounded, continuous activation function. This particular parametrization is called a two-layer neural network (with one hidden layer). In order to train the optimal parameters, one need to solve the optimization problem:
\[\inf_{(c_i, a_i, b_i)_{1\le i\le N}} \sum_{j=1}^M L\left(f(x_j), \sum_{i=1}^N c_i \f(a_i\cdot x_j + b_i)\right),\]
where $L: (y,z) \mapsto L(y,z)$ is a loss function, typically convex in $z$. Here we face an overparametrized, non-convex optimization, and have no theory for an efficient solution. However it has been recently observed (see e.g. \cite{mei2019mean, HRSS19, chizat2018global, KRTY20}) that by lifting the optimization problem to the space of probability measures, namely   
\[\inf_{p\in \cP}\sum_{j=1}^M L\Big(f(x_j), \dbE^p[C\f(A\cdot x +B)] \Big), \]
with random variables $(C, A, B)$ taking values in $\dbR\times\dbR^d\times\dbR$ following the distribution $p$, one makes the optimization convex (the function $F: p\mapsto \sum_{j=1}^M L\big(f(x_j), \dbE^p[C\f(A\cdot x +B)] \big) $ is convex), and has extensive tools to find the minimizers. 

Unlike in \cite{chizat2018global} where the authors address the mean field optimization directly, in \cite{mei2019mean, HRSS19} the authors add the entropy regularizer $H(p):=\int p(x) \log p(x) dx $, that is, they aim at solving the regularized optimization problem:
\begin{equation}\label{eq:intro_entropy_reg}
\inf_{p\in \cP} F(p) + \frac{\si^2}{2} H(p).
\end{equation}
Recall the definition of the linear derivative $\frac{\d F}{\d p}$ and the intrinsic derivative $D_p F$ (see \Cref{rem:intrinsic derv} below) for functions on the space of probability measures. In \cite{HRSS19} the authors introduce the mean field Langevin (MFL for short) dynamics:
\[dX_t = - D_p F(p_t, X_t) dt + \si dW_t,\]
where $p_t = {\rm Law}(X_t)$ and $W$ is a standard Brownian motion, and prove that the marginal laws $(p_t)_{t\ge 0}$ of the MFL dynamics converge towards the minimizer of the entropic regularization \eqref{eq:intro_entropy_reg}. In the following works \cite{NWS22, Chizat22} it has been shown that the convergence is exponentially quick. 

In this paper we try to look into the mean field optimization problem from another perspective, by adding the Fisher information $I(p):= \int |\nabla \log p(x)|^2 p(x)dx$ instead of the entropy as the regularizer, namely solving the regularized optimization
\[\inf_{p\in \cP} \fF^\si(p)  , \quad \fF^\si(p):= F(p) +\frac{\si^2}{4}I(p).\]
By convexity and calculus of variation (see \Cref{thm:foc}), it is not hard to see that $p^*\in \argmin_{p\in \cP} \fF^\si(p)$ if 
\begin{equation}\label{eq:intro-FOE}
    \frac{\d \fF^\si}{\d p}(p^*, x) := \frac{\delta F}{\delta p} \left(p^*, x\right) - \frac{\sigma^2}{4} \left( 2 \Delta \log p^* + \left|\nabla\log p^*\right|^2 \right)= {\rm constant}.
\end{equation} 
We shall introduce the mean field Schr\"odinger (MFS for short) dynamics:
\[\partial_t p_t = -  \frac{\d \fF^\si}{\d p}(p_t, \cdot) p_t, \]
prove its wellposedness and show that its marginal distributions $(p_t)_{t\ge 0}$ converges (uniformly) towards the minimizer of the free energy function $\fF^\si$. One crucial observation is that the free energy function decays along the MFS dynamics:
\[\frac{d \fF^\si(p_t)}{dt} =  - \int \left|\frac{\d \fF^\si}{\d p}(p_t, x)\right|^2 p_t(dx).\]
In order to prove it rigorously, we develop a probabilistic argument (coupling of diffusions) to estimate $(\nabla \log p_t,\nabla^2 \log p_t)_{t\ge 0}$. Remarkably, the estimate we obtain is uniform in time. Using the energy dissipation we can show that $(p_t)_{t\ge 0}$ converges  exponentially quickly with help of the convexity of $F$ and the Poincar\'e inequality. Another main contribution of this paper is to show that the MFS dynamics is a gradient flow of the free energy function $\fF^\si$ on the space of probability measures, provided that the `distance' between the probability measures is measured by relative entropy. Finally it is noteworthy that MFS dynamics  is numerically implementable, and we shall briefly propose a Monte Carlo simulation method. 

\paragraph{Related works.} Assume $F$ to be linear, i.e. $F(p):= \int f(x) p(dx)$ with a real potential function $f$ and denote the wave function by $\psi:= \sqrt{p}$. Then the function $\fF^\si$ reduces to the conventional energy function in quantum mechanics, composed of the potential energy $\langle \psi, f \psi\rangle_{L^2}$ and the kinetic energy $\si^2\langle \nabla \psi,\nabla \psi\rangle_{L^2}.$ Meanwhile, the MFS dynamics is reduced to the semigroup generated by the Schr\"odinger operator:
\begin{equation}\label{eq:intro-Schr\"odinger eq}
\partial_t \psi = - \mathcal{H} \psi, \quad\text{with } \mathcal{H}:= -\frac{\si^2}{2} \Delta + \frac12 f.
\end{equation}
The properties of the classical Schr\"odinger operator, including its longtime behavior, have been extensively studied in the literature, see e.g. the monographs \cite{RS1234, Lewin22}.
There are also profound studies in  cases where $F$ is nonlinear, notably the density functional theory \cite{ED11, Eschrig}.  However, to our knowledge there is no literature dedicated to the category of convex potential $F:\cP\rightarrow \dbR$, and studying the longtime behavior of such nonlinear Schr\"odinger operator by exploiting the convexity. In addition, the probabilistic nature of our arguments seems novel. 

Using the change of variable: $u:=-\log p^*$, the first order equation \eqref{eq:intro-FOE} can be rewritten as
\[\frac{\si^2}{2}\Delta u -\frac{\si^2}{4}|\nabla u|^2 + \frac{\d F}{\d p}(p^*, x)=\mbox{\rm constant}.\]
So the function $u$ solves an ergodic Hamilton-Jacobi-Bellman equation, and its gradient $\nabla u$ is the optimal control for the ergodic stochastic control problem:
\begin{gather*}
    \lim_{T\rightarrow\infty} \frac1T \sup_\a \dbE\left[\int_0^T \left(\frac12|\a_t|^2 + \frac{2}{\si^2}\frac{\d F}{\d p}(p^*, X^\a_t)\right)dt\right], 
    \intertext{where}
     dX^\a_t = \a_t dt + \sqrt{2}dW_t.
\end{gather*}
Further note that the probability $p^*=e^{-u}$ coincides with the invariant measure of the optimal controlled diffusion: $dX^*_t = -\nabla u(X^*_t)dt + \sqrt{2}dW_t$, so that $p^*$ is the Nash equilibrium of the corresponding ergodic mean field game. For more details on the ergodic mean field game, we refer to the seminal paper \cite{lasry2007}, and for more general mean field games we refer to the recent monographs \cite{CarmonaDelarueMFGBook1, CarmonaDelarueMFGBook2}. Our convergence result of the MFS dynamics $(p_t)_{t\ge 0}$ towards $p^*$ offers an approximation to the equilibrium of the ergodic mean field game. 

Our result on the gradient flow, as far as we know, is new to the literature. It is well known to the community of computational physics that the normalized solution $(\psi_t)_{t\ge 0}$ to the imaginary time Schr\"odinger equation \eqref{eq:intro-Schr\"odinger eq} is the gradient flow of the free energy $\fF^\si$ on the $L^2$-unit ball. On the other hand, in \cite{RJBV19} the authors discuss the (linear) optimization problem without Fisher information regularizer, and formally show that the dynamics,
\(\partial_t p  = - f p,\)
is the gradient flow of the potential functional $\int f dp$ on the  space of probability measures provided that the distance between the measures are measured by the relative entropy. Inspired by these works, we prove in the current paper that the solution to the variational problem:
\[ p^h_{i+1} : = \argmin_{p\in\cP} \left\{ \fF^\si(p) + h^{-1} H(p | p^h_i)\right\}, \quad\mbox{for $h>0, ~i\ge 0$},\]
converges to the continuous-time flow of the MFS dynamics as $h\rightarrow 0$. This result can be viewed as a counterpart of seminal paper \cite{JKO98} on the Wasserstein-$2$ gradient flow. 
\vspace{5mm}

The rest of the paper is organized as follows. In \Cref{sec:main result} we formulate the problem and state the main results of the paper. The proofs are postponed to the subsequent sections. In \Cref{sec:mfSchr\"odinger}, we show that the MFS dynamic is well-defined and admits an important decomposition as the exponential of a sum of a convex and a Lipschitz function. Then we study the long time behavior of this dynamic in \Cref{sec:mfconvergence}  and we prove that it converges exponentially fast to the unique minimizer of the mean field optimization problem regularized by Fisher information.  Finally we establish in \Cref{sec:gradientflow}  that the MFS dynamic corresponds to the gradient flow with respect to the relative entropy. Some technical results including a refined reflection coupling result are also gathered in Appendix.

\paragraph{Notations} 
(i) For each \(T >0\), we denote by \(Q_T = \left(0, T\right] \times \mathbb R^d, ~\bar Q_T = \left[0, T\right] \times \mathbb R^d\) and by $C^n(Q_T)$ the set of functions $f$ such that $\partial_t^k\nabla^m f$ is continuous on $Q_T$ for $2k+m\le n$.  In the case $T=+\infty,$ we simply write \(Q = \left(0, \infty\right) \times \mathbb R^d, ~\bar Q = \left[0, +\infty\right) \times \mathbb R^d.\)\\
(ii) Given a measure $\mu$ on $\dbR^d,$ let  $W^{k,p}(\mu)$ be the Sobolev space of functions $f:\dbR^d\to\dbR$ such $f\in L^p(\mu)$ and $\nabla^l f\in L^p(\mu)$ for all $l\leq p.$ In particular, we denote $H^1(\mu) := W^{1,2}(\mu).$ We simply write $W^{k,p}$ and $H^1$ when $\mu$ is the Lebesgue measure. \\
(iii) Let $\cP_p(\dbR^d)$ be the collection of distribution on $\dbR^d$ with finite first $p$ moments. It is equipped with $\cW_p$ the Wasserstein distance of order $p.$ \\
(iv) Given $u:\dbR^d\to\dbR,$ we consider the functional norms $\| u \|_{(2)}:= \sup_{x\in \dbR^d} \frac{|u(x)|}{1+|x|^2}$  and $\| u \|_{\infty}:= \sup_{x\in \dbR^d} |u(x)|.$

\section{Main Results}\label{sec:main result}

\subsection{Free Energy with Fisher Information}

Denote by $\cP_2(\dbR^d)$ the set of all probability measures on $\dbR^d$ with finite second moments, endowed with $\cW_2$ the Wasserstein distance of order $2$. 
We focus on the probability measures admitting densities, and denote the density of $p\in \cP_2(\dbR^d)$ still by $p : \dbR^d \rightarrow \dbR$ if it exists. In particular we are interested in the probability measures of density satifying:
\[
\cP_H := \left\{ p\in \cP_2(\dbR^d): ~ \sqrt{p} \in H^1 \right\}.
\]
In this paper we study a regularized mean field  optimization problem, namely, given a potential function $F:\cP_2(\dbR^d) \rightarrow \dbR$ we aim at solving
\begin{equation}
\label{eq:definition-free energy}
\inf_{p \in \cP_H } ~ \fF^\si(p), \q\mbox{with} \q \fF^\si(p): = F(p) + \si^2 I (p),
\end{equation}
where $\sigma>0$ and  $I$ is the Fisher information defined by
\begin{equation}
 I(p) := \int_{\dbR^d} |\nabla \sqrt p(x)|^2 dx.
\end{equation}
In the literature, $\fF^\si$ is called the Ginzburg--Landau energy function with temperature $\si$. Note that for $p\in \cP_H$ and $p>0,$ it holds
\[
4 \int_{\dbR^d} |\nabla \sqrt p(x)|^2 dx= \int_{\dbR^d} \left| \nabla \log p(x) \right|^2 p(x)d x.
\]
Throughout the paper, we assume that the potential function $F$ is smooth, convex and coercive as stated in the following assumption.

\begin{defn}
We say that a function $F: \cP_2(\dbR^d)\rightarrow\dbR$ is $\cC^1$ if there exist $\frac{\d F}{\d p}: \cP_2(\dbR^d)\times \dbR^d \rightarrow\dbR$ continuous with quadratic growth in the second variable such that for all $p, q\in \cP_2(\dbR^d),$
\[F(q)-F(p) = \int_0^1 \int_{\dbR^d} \frac{\d F}{\d p}\big(t q+(1-t)p, x\big) (q-p)(dx) dt. \]
\end{defn}

\begin{rem}\label{rem:intrinsic derv}\
Note that $F\in \cC^1$ is $\cW_2$-continuous and $\frac{\d F}{\d p}$ is defined up to constant. We call $\frac{\d F}{\d p}$ the linear derivative and we may further define the intrinsic derivative $D_p F(p,x):= \nabla \frac{\d F}{\d p}(p,x)$.  
\end{rem}

\begin{assum}\label{assum:potential}
Assume that $F$ is $\cC^1,$ convex and
\[
F(p) \ge \l \int_{\dbR^d} |x|^2 p(dx) \q \mbox{for some $\l >0$.}
\]
\end{assum}

The following proposition states that the bias caused by the regularizer vanishes as the temperature $\si\rightarrow 0$. It ensures that the Fisher information is efficient as regularizer in this mean field  optimization problem.

\begin{prop}
\label{thm:gamma-convergence}
It holds
\[
\lim_{\si\rightarrow 0} \left( \inf_{p \in \cP_H } ~ \fF^\si(p) \right) = \inf_{p \in \mathcal P_2 } ~ F(p).
\]
\end{prop}

\begin{proof}
Given $\varepsilon>0,$ let \(p \in \mathcal P_2\) be such that \(F\left(p\right) < \inf_{p \in \mathcal P_2} F\left(p\right) +
\varepsilon\). 
By truncation and mollification, define $p_{K,\d} := p_K * \varphi_\delta$ where \(p_K := \frac{p \mathbbm 1_{|x| \leq K}}{p(|x| \leq K)}\) and  \(\varphi_\delta (x) := \frac{1}{(2\pi \delta)^{\frac d2}} \exp (- \frac{|x|^2}{2\delta})\). 
It is clear that $p_{K,\d}$ converges to $p$ in $\cW_2$ as $K\to\infty$ and $\delta\to 0.$ Additionally, one easily checks by direct computation that $I(p_{K,\d})<+\infty.$
By $\cW_2$--continuity of $F,$ we deduce by choosing $K$ large and $\delta$ small enough that 
\begin{equation*}
\inf_{p \in \mathcal P_H} \mathfrak F^\sigma \left(p\right) \leq F\left(p_{K,\d}\right) + \frac{\sigma^2}{2} I\left(p_{K,\d}\right) \leq F\left(p\right) + \varepsilon+ \frac{\sigma^2}{2} I\left(p_{K,\d}\right)  \leq \inf_{p \in \mathcal P_2} F\left(p\right) + 2 \varepsilon + \frac{\sigma^2}{2} I\left(p_{K,\d}\right).
\end{equation*}
 We conclude by taking the limit \(\sigma\to 0\). 
\end{proof}

For the gradient flow analysis of \Cref{sec:gradient} below,  we shall actually consider a slightly more general mean field optimization problem.  
Namely, we aim at minimizing the following generalized free energy function: for all $p\in\cP_H,$
\begin{equation}\label{eq:general free energy}
 \fF^{\si, \gamma}(p) := F(p) + \si^2 I (p) + \gamma H(p),
\end{equation}
where $\gamma\geq 0$ and $H$ is the entropy defined as \[H(p):= \int_{\dbR^d} p(x)\log p(x) dx.\] 
By considering the limit of the rate of change $\frac{\fF^{\sigma, \gamma}(p+t(q-p))- \fF^{\sigma, \gamma}(p)}{t}$ as $t\to 0, $ a formal calculus leads to define by abuse of notation
\begin{equation}\label{eq:dFsigmadp}
\frac{\delta \fF^{\sigma, \gamma}}{\delta p}(p,\cd) ~ : =~ \frac{\delta F}{\delta p} \left(p, \cd \right) - \frac{\sigma^2}{2} \Delta \log p - \frac{\sigma^2}{4} \left|\nabla \log p\right|^2 + \gamma \log p - \lambda(p),
\end{equation}
where $\lambda(p)\in\dbR$ is chosen so that
\begin{equation}\label{eq:normaldFsigmadp}
 \int_{\dbR^d} \frac{\delta \fF^{\sigma,\gamma}}{\delta p}(p,x) p(x) dx =0.
\end{equation}
The details of this calculation can be found within the proof of \Cref{thm:foc} below. 
Note also that equivalent formulas for $\frac{\delta \fF^{\sigma, \gamma}}{\delta p}$ can be obtained by observing that 
\begin{equation*}
 \Delta \log p + \frac{1}{2} \left|\nabla \log p\right|^2 = \frac{\Delta p}{p} - \frac{1}{2} \frac{\left|\nabla p\right|^2}{p^2} = 2 \frac{\Delta \sqrt{p}}{\sqrt{p}}.
\end{equation*}

\subsection{Mean Field Schr\"odinger Dynamics}

Given the definition in \cref{eq:dFsigmadp}, we will consider the following generalized mean field  Schr\"odinger (MFS for short) dynamics
\[\partial_t p_t = - \frac{\delta \fF^{\sigma, \gamma}}{\delta p} \left(p_t, \cdot\right) p_t.\]
Thanks to the normalization in \cref{eq:normaldFsigmadp}, the mass of \(p_t\) is conserved to \(1\).
Writing the functional derivative explicitly, we have the following dynamics
\begin{equation}
\label{eq:gradient-flow}
\partial_t p_t = - \left( \frac{\delta F}{\delta p} \left(p_t, \cdot\right) - \frac{\sigma^2}{2} \Delta \log p_t - \frac{\sigma^2}{4} \left|\nabla \log p_t\right|^2 + \gamma \log p_t - \lambda_t \right) p_t
\end{equation}
where \(p_t = p\left(t, \cdot\right)\) and \(\lambda_t = \lambda(p_t)\) satisfies
\[
\lambda_t = \int_{\dbR^d} \left( \frac{\delta F}{\delta p} \left(p_t, x\right) - \frac{\sigma^2}{2} \Delta \log p_t (x) - \frac{\sigma^2}{4} \left|\nabla \log p_t (x)\right|^2 + \gamma \log p_t(x) \right) p_t(x) dx.
\]
In particular, the important case $\gamma=0$ is called the MFS dynamics, namely,
\begin{equation}\label{eq:mfSchr\"odinger}
 \partial_t p_t = - \frac{\delta \fF^{\sigma}}{\delta p} \left(p_t, \cdot\right) p_t.
\end{equation}

Intuitively the generalized MFS dynamics follows the direction of steepest descent as it moves in the opposite direction of the derivative $\frac{\delta \fF^{\sigma, \gamma}}{\delta p}.$ 
To ensure that it is indeed converging towards a minimizer of $\fF^{\sigma, \gamma},$  
the crucial assumption in this paper is that the derivative $\frac{\delta F}{\delta p}$ decomposes into the sum of a convex potential and a Lipschitz perturbation as stated below.

\begin{assum}
\label{assum:decomposition-F-derivative}
The linear derivative admits the decomposition \(\frac{\delta F}{\delta p} \left(p, x\right) = g\left(x\right) + G\left(p, x\right) \) where $g$ and $G(p,\cdot)$ are $C^2$ such that
\begin{itemize}
 \item[\rm (i)] \(g\) is $\ul \kappa$--convex and has bounded Hessian, \ie,
 \[ \ul\kappa I_d \le \nabla^2 g \le \ol\kappa I_d, \quad\mbox{for some \(\ol\kappa \ge \ul \kappa\)}.\]

 \item[\rm (ii)] $G$ is $\cW_1$--continuous in $p$ and Lipschitz continuous in $x$, \ie,  for all $x,y\in \dbR^d, \, p\in\cP_2(\dbR^d),$
 \[ | G(p,x) - G(p,y)| \le L_G |x -y|.\]

 \item[\rm (iii)] \(\nabla G\) is Lipschitz continuous, \ie, for all $x,y\in \dbR^d, \, p,q\in \cP_2(\dbR^d),$
 \[ |\nabla G(p,x) - \nabla G(q,y)| \le L_{G}
 \left(|x -y| + \cW_1(p,q)\right).\]

\end{itemize}
\end{assum}

\begin{assum}\label{assum:initialdistribution}
The initial distribution admits the decomposition $p_0(x) = e^{-(v_0(x) + w_0(x))}$ where $v_0$ and $w_0$ are $C^1$  such that
\begin{itemize}

 \item[\rm (i)] $v_0$ is $\ul\eta_0$--convex and $\nabla v_0$ is Lipschitz continuous,  \ie, for all $x,y\in \dbR^d,$
 \[  | \nabla v_0(x) - \nabla v_0(y)| \le \ol\eta_0 |x -y|,\quad \left(\nabla v_0(x) - \nabla v_0(y)\right)\cdot (x-y) \ge \ul\eta_0 |x-y|^2.\]
 \item[\rm (ii)] $w_0$ and $\nabla w_0$ are both Lipschitz continuous, \ie, for all $x,y\in \dbR^d,$
 \[  | w_0(x) - w_0(y)| + |\nabla w_0(x) - \nabla w_0(y)| \le L_0 |x -y|.\]
\end{itemize}
\end{assum}

In the sequel, we assume that Assumptions~\ref{assum:potential}, \ref{assum:decomposition-F-derivative} and \ref{assum:initialdistribution} hold. First we show that the generalized MFS dynamic is well-defined and that it decomposes as the exponential of a sum of a convex and a Lipschitz function.  The proof is postponed to \Cref{sec:proof-wellposedness}.

\begin{thm}\label{thm:wellposedness}
Under the assumptions above, the generalized MFS dynamics \cref{eq:gradient-flow} admits a unique positive classical solution $p\in C^3(Q)\cap C(\bar Q).$
In addition,  it admits the decomposition $p_t = e^{-(v_t + w_t)}$ where there exist $\ul \eta, \ol \eta, L>0,$ such that
 \begin{equation}\label{eq:wellposedness}
  \ul \eta I_d \le \nabla^2 v_t \le \ol \eta I_d, \qquad \| \nabla w_t\|_{\infty}\vee \| \nabla^2 w_t\|_{\infty} \le L,\qquad \forall\,t > 0.
 \end{equation}
\end{thm}

Then we study the long-time behaviour of the generalized MFS dynamics and  establish convergence toward the unique minimizer of the generalized free energy function. The proof is postponed to \Cref{sec:convergence}.  It essentially relies on energy dissipation which can be derived formally as follows:
\begin{equation*}
 \frac{d}{dt} \fF^{\si, \gamma} (p_t) =  \int_{\dbR^d} \frac{\d \fF^{\si, \gamma}}{ \d p}(p_t, x) \partial_t p_t(x) dx = - \int_{\dbR^d} \left|\frac{\d \fF^{\si, \gamma}}{ \d p}(p_t, x)\right|^2 p_t(x) dx,
\end{equation*}
See \Cref{thm:energydecrease} below for a proof. 
It follows that the generalized free energy monotonously decreases along the generalized MFS dynamics \eqref{eq:gradient-flow}. Intuitively, the dissipation of energy only stops at the moment \(\frac{\d \fF^{\si, \gamma}}{ \d p}(p^*, \cdot) = 0\).  Since $\fF^{\si, \gamma}$ is (strictly) convex, it is a sufficient condition for $p^*$ to be the minimizer, see  \Cref{thm:foc} below.
  
\begin{thm}\label{thm:convergence}
Under the assumptions above, the solution $(p_t)_{t\ge 0}$ to~\cref{eq:gradient-flow}  converges uniformly on $\dbR^d$ to $p^*$, the unique minimizer of $\fF^{\si, \gamma}$ in $\cP_H$. 
In addition, the optimizer $p^*$ satisfies~\cref{eq:wellposedness} and it is a stationary solution to \cref{eq:gradient-flow}, \ie,
\begin{equation}\label{eq:mfs-foc}
\frac{\d \fF^{\si, \gamma}}{ \d p}(p^*, \cdot) = 0.
\end{equation}
\end{thm}

\begin{rem}
 By \Cref{lem:property-p-uniform} below,  the family of distributions $(p_t)_{t\geq 0}$ admits uniform Gaussian bounds and thus it also converges to $p^*$ for the $L^p$--norm or the $\cW_p$--distance for any $p\geq 1.$
\end{rem}

\begin{rem}\label{rem:linear case}
In case that the function $p\mapsto F(p)$ is linear, \ie, $F(p) = \int_{\dbR^d} f(x) p(dx)$ with some potential $f$, the function $\fF^\si$ is the classical energy function in quantum mechanics composed of the potential energy $F$ and the kinetic one $\int_{\dbR^d} |\nabla \sqrt p(x)|^2 dx $. Let $p^*$ be the minimizer of $\fF^\si$, and denote by $\psi^*:=\sqrt{p^*}$ the corresponding wave function. 
Then the first order equation~\eqref{eq:mfs-foc} reads
\[-\si^2 \Delta \psi^* + f \psi^* = c \psi^*, \quad\mbox{with}\quad c = \fF^\si(p^*) = \min_{p\in \cP_H} \fF^\si(p). \]
It is well known that $c$ is the smallest eigenvalue of the Schr\"odinger operator $-\si^2\Delta + f$ and that $\psi^*$ is the ground state of the quantum system.
\end{rem}

Further we shall prove that the convergence for the MFS dynamics (with $\gamma=0$) is exponentially quick. See~\Cref{sec:exp-cvg} below for a proof. As a byproduct,  we establish a functional inequality in \Cref{thm:functional_ineq} which may carry independent interest.

\begin{thm}\label{thm:exp_convergence}
There exists a constant \(c(\ul\eta,\ol\eta, L,d,\sigma) >0\) such that
\begin{equation}\label{eq:exp_convergence_energy}
    \fF^\sigma(p_t) - \fF^\sigma(p^*) \leq e^{-ct} (\fF^\sigma(p_0) - \fF^\sigma(p^*)). 
\end{equation}
Moreover,  it holds 
\[\frac{\si^2}{4}I(p_t| p^*) \le e^{-ct} (\fF^\sigma(p_0) - \fF^\sigma(p^*)),\]
where $I(p_t| p^*):= \int p_t |\nabla\log(p_t/p^*)|^2$ is the relative Fisher information.
\end{thm}

\subsection{Gradient Flow with Relative Entropy}
\label{sec:gradient}

In this paper, we shall further investigate the gradient flow of the free energy function $\fF^\si$ with respect to the relative entropy.
 First, given $h>0$ and a distribution $\tilde p$ satisfying  \Cref{assum:initialdistribution}, consider the variational problem:
\begin{equation}\label{eq:mfo-entropy}
 \inf_{p\in \cP_H} \left\{ \fF^\si(p) + h^{-1} H(p | \tilde p)\right\}.
\end{equation}
where $H(p|\tilde p):=\int p \log(p/\tilde p)$ is the relative entropy.
In view of \Cref{assum:initialdistribution}, we have the decomposition $\tilde p = e^{-\tilde u}$ with $\tilde u = \tilde v + \tilde w$. Denoting by
\[\tilde F(p):= F(p) + h^{-1} \int_{\dbR^d}\tilde u (x) p(dx)\]
the new potential function, we may rewrite the objective function in the optimization \eqref{eq:mfo-entropy} in the form of the generalized free energy function \eqref{eq:general free energy}, \ie 
\[\tilde\fF^{\si, h^{-1}}(p) = \tilde{F}(p) + \si^2 I (p) + h^{-1} H(p).\]
Moreover, the new potential function $\tilde F$ still satisfies \Cref{assum:decomposition-F-derivative} with $\tilde g = g + h^{-1}\tilde v$ and $\tilde G = G + h^{-1}\tilde w$. 
Therefore, the following result is a straightforward consequence of \Cref{thm:convergence}.

\begin{cor}\label{cor:existence variational form}
If $\tilde p$ satisfies \Cref{assum:initialdistribution}, the minimization problem \eqref{eq:mfo-entropy} admits a unique minimizer $p^*\in\cP_H$ still satisfying \Cref{assum:initialdistribution} (with different coefficients) and it satisfies the first order condition  \[\frac{\d \tilde{\fF}^{\si, h^{-1}}}{ \d p}(p^*, \cdot) = 0.\]
\end{cor}

Now given $p^h_0:= p_0 $ satisfying \Cref{assum:initialdistribution}, we may define a sequence of probability measures using the variational problem \eqref{eq:mfo-entropy}:
\begin{equation}\label{eq:discrete GD}
 p^h_i : = \argmin_{p\in \cP_H} \left\{ \fF^\si(p) + h^{-1} H(p | p^h_{i-1})\right\}, \quad\mbox{for $i\ge 1$}.
\end{equation}
It corresponds to the so--called \emph{minimizing movement scheme} in the optimal transport literature.
According to \Cref{cor:existence variational form}, the minimizer $p^h_i$ is well-defined and it satisfies the first order condition:
\begin{equation}\label{eq:p^h_i}
\frac{\delta \fF^\si}{\delta p} \left(p^h_i, \cdot \right) + h^{-1} (\log p^h_i - \log p^h_{i-1}) = \int  h^{-1} (\log p^h_i - \log p^h_{i-1}) p^h_i. 
\end{equation}
Thus we expect as $h\to 0$ that the minimizing movement scheme $p^h$ converges to the corresponding gradient flow $p$ satisfying
\begin{equation*}
  \frac{\delta \fF^{\sigma}}{\delta p} \left(p_t, \cdot\right) + \partial_t \log p_t = 0,
\end{equation*}
which corresponds to the MFS dynamics~\eqref{eq:mfSchr\"odinger}.

This result is proved rigorously in \Cref{sec:gradientflow-proof} below.  By slightly abusing the notations,  define the continuous-time flow of probability measures:
\begin{equation*}
     p^h_t : =  p^h_i, \quad\mbox{for $t\in [hi, h(i+1))$}.
\end{equation*}

\begin{thm}\label{thm:gradientflow}
The sequence of functions
\((p^h)_{h>0}\) converges, uniformly on $[0,T]\times\dbR^d$ for any $T>0,$  to $p$ the MFS dynamics~\eqref{eq:mfSchr\"odinger}.
\end{thm}

\begin{rem}
 In view of \Cref{cor:p properties} below,  the family of distributions $(p^h)_{h>0}$ admits uniform Gaussian bounds and thus we also have for any $p\geq 1,$
\begin{equation*}
 \sup_{t\in[0,T]} \big\|p_t^h - p_t\big\|_{L^p}\xrightarrow[h\to 0]{} 0, \qquad \sup_{t\in[0,T]} \cW_p\big(p_t^h, p_t\big)\xrightarrow[h\to 0]{} 0.
\end{equation*} 
\end{rem}

\subsection{Numerical Simulation}\label{sec:algorithm}

In this section we shall briefly report how to sample $\frac{1}{N}\sum_{i=1}^{N} \d_{X^i_t}$ to approximate the probability law $p_t$ in the MFS dynamics \eqref{eq:mfSchr\"odinger}, without pursuing mathematical rigorism.

Observe first that the MFS dynamics \eqref{eq:mfSchr\"odinger} can be rewritten as
\[\partial_t p_t = \frac{\si^2}{2} \Delta p_t - \left( \frac{\delta F}{\delta p} \left(p_t, \cdot\right) + \frac{\sigma^2}{4} \left|\nabla\log p_t\right|^2 - \lambda_t \right) p_t.\]
This can viewed as the Fokker--Planck equation describing the marginal distribution of a Brownian motion $(X_t)_{t\ge 0}$ killed at rate $ \eta(t,x):= \frac{\delta F}{\delta p} (p_t, x) + \frac{\sigma^2}{4}\left|\nabla \log p_t (x)\right|^2 $ conditionned on not being killed. 
In other words,
the particle $X$ moves freely in the space $\dbR^d$ as a Brownian motion $(\si W_t)_{t\ge 0}$ before it gets killed with conditional probability
\[\dbP\left(~\mbox{X gets killed in}~[t, t+\Delta t]~| ~ X_t\right)\approx \eta(t, X_t)\Delta t , \quad\mbox{for small $\Delta t$}.\]
Meanwhile the killed particle gets reborn instantaneously according to the distribution $p_t$. This interpretation of the MFS dynamics offers an insight on how to sample the marginal law $p_t.$ However, in order to evaluate the death rate $\eta(t, X_t),$ one needs to evaluate $|\nabla \log p_t|^2$, which can be hard if not impossible in practice. 
This difficulty forces us to find a more sophisticated way to sample $p_t$.

Now observe that $\psi_t := \sqrt{ p_t}$ solves the PDE:
\begin{equation}\label{eq:mfSchr\"odinger-root}
\partial_t \psi_t = \frac{\si^2}{2}\Delta \psi_t - \frac12 \left(\frac{\d F}{\d p}(p_t, \cdot) -\lambda_t\right)\psi_t. 
\end{equation}
Then introduce two scaling of $\psi_t$, namely, $\bar \psi_t := e^{- \frac12\int_0^t \lambda_s \,ds} \psi_t$ and $\hat \psi_t := (\int \psi_t)^{-1} \psi_t$ so that
\[\partial_t \bar \psi_t = \frac{\si^2}{2}\Delta \bar \psi_t - \frac12 \frac{\d F}{\d p}(p_t, \cdot)\bar \psi_t,\qquad
\partial_t \hat \psi_t = \frac{\si^2}{2}\Delta \hat \psi_t - \frac12 \left(\frac{\d F}{\d p}(p_t, \cdot)- \hat \lambda_t\right) \hat \psi_t,\]
where the constant $\hat \lambda_t\in\dbR$ is chosen so that $\hat\psi_t$ is a probability density. Observe that:
\begin{itemize}
 \item By the Feynman--Kac formula, the function $\bar\psi$ has the probabilistic representation:
 \begin{align*}
 \bar\psi_t(x) &= \dbE\left[\exp\Big(-\int_0^t \frac12 \frac{\d F}{\d p}(p_{t-s}, x+ \si W_s)ds\Big)\psi_0(x+\si W_t)\right]\\
 & \approx \frac{1}{M}\sum_{j=1}^M \exp\Big(-\int_0^t \frac12 \frac{\d F}{\d p}(p_{t-s}, x+ \si W^j_s)ds\Big)\psi_0(x+\si W^j_t),
 \end{align*}
 where the latter is the standard Monte Carlo approximation of the expectation.
 
  \item The probability law $\hat \psi$ is the marginal distribution of a Brownian motion killed at rate  $\eta(t,x):=\frac12 \frac{\d F}{\d p} (p_t, x)$ conditioned on not being killed. 
  It can be sampled by simulating a large number $(\hat X^i)_{1\le i\le N}$ of independent Brownian particles killed at rate $\eta$ which upon dying are instantaneously reborn by duplicating one of the living particles. 

 \item Eventually,  the distribution $p_t$ can be approximately sampled as the following weighted empirical measure
 \[p_t  = \frac{\bar\psi_t }{\int_{\dbR^d} \bar\psi_t(x)\hat\psi_t(x)dx}\hat\psi_t \approx \frac{1}{N} \sum_{i=1}^{N}\frac{\bar\psi(t, \hat X^i_t) }{\frac{1}{N} \sum_{k=1}^{N} \bar\psi(t, \hat X^k_t)} \d_{\hat X^i_t}.\]
\end{itemize}

\begin{rem}
In particular, in view of \Cref{rem:linear case}, the Monte Carlo method above offers an efficient way to sample the ground state of a high dimensional quantum system. To our knowledge there is little discussion on similar numerical schemes in the literature.
\end{rem}

\section{Mean Field Schr\"odinger Dynamics}
\label{sec:mfSchr\"odinger}

In order to study the generalized MFS dynamics in \cref{eq:gradient-flow},  we introduce a change of variable $p_t = e^{-u_t} / \int e^{-u_t}$ where $u$ satisfies  the
following equation:
\begin{equation}
\label{eq:HJB}
\partial_t u_t = \frac{\sigma^2}{2} \Delta u_t - \frac{\sigma^2}{4} \left|\nabla u_t\right|^2 + \frac{\delta F}{\delta p} \left(p_t, \cdot\right) - \gamma u_t,
\end{equation}
with initial condition $u_0=-\log{p_0}.$
Clearly,  $u$ is a classical solution to \cref{eq:HJB} if and only if the probability density \(p\) is a positive classical solution to \cref{eq:gradient-flow}.
Thus we consider the mapping
\begin{equation}\label{eq:contraction}
(m_t)_{t\in [0,T]} \mapsto (u_t)_{t\in [0,T]} \mapsto (p_t)_{t\in [0,T]}
\end{equation} 
where $p_t = e^{-u_t} / \int e^{-u_t}$ and $u$ solves the equation
\begin{equation}\label{eq:fixpoint_m-to-u}
 \partial_t u_t = \frac{\sigma^2}{2} \Delta u_t - \frac{\sigma^2}{4} \left|\nabla u_t\right|^2 + \frac{\delta F}{\delta p} \left(m_t, \cdot\right) -\gamma u_t,
\end{equation}
 and we look for a fixed point to this mapping as it corresponds to a solution to \cref{eq:HJB}. 
 Note that \cref{eq:fixpoint_m-to-u} corresponds to the Hamilton--Jacobi--Bellman (HJB for short) equation of a classical linear-quadratic stochastic control problem and so $u$ is well-defined as the unique viscosity solution of this equation by standard arguments.

In this section, we first show that the solution to the HJB equation \eqref{eq:fixpoint_m-to-u} can be decomposed as the sum of a convex and a Lipschitz function. This allows us to apply a reflection coupling argument to show that the mapping~\eqref{eq:contraction} is a contraction on short horizon and thus to ensure existence and uniqueness of the solution to \eqref{eq:HJB}.This completes the proof of \Cref{thm:wellposedness}. Finally we gather some properties of the solution to \cref{eq:gradient-flow} for later use.

\subsection{Hamilton--Jacobi--Bellman Equation}
\label{sec:decomposition}

The aim of this section is to prove that the solution to the HJB equation \eqref{eq:fixpoint_m-to-u} is smooth and can be decomposed into the sum of a convex and a Lipschitz function as stated in Proposition~\ref{prop:hjb-main} below.  Throughout this section we assume that the following assumption holds.

\begin{assum}\label{assum:mLipschitz}
 Assume that the mapping $t\mapsto m_t$ is $\cW_1$--continuous, \ie,
 \[\lim_{s\rightarrow t}\cW_1(m_t, m_s) =0, \quad\mbox{for all $t\ge 0$}.\]
\end{assum}

\begin{prop}\label{prop:hjb-main}
 There exists a unique classical solution $u\in C^3(Q)\cap C(\bar Q)$ to the HJB equation~\eqref{eq:fixpoint_m-to-u}. 
 In addition,  $u=v+w$ where there exist $\ul \eta, \ol \eta, L>0,$ independent of $m,$ such that
 \begin{equation*}
  \ul \eta I_d \le \nabla^2 v_t \le \ol \eta I_d, \qquad \| \nabla w_t\|_{\infty}\vee \| \nabla^2 w_t\|_{\infty} \le L,\qquad \forall\,t > 0.
 \end{equation*}
\end{prop}

By the Cole--Hopf transformation, we may prove in a rather classical way that there exists a unique smooth solution to \cref{eq:fixpoint_m-to-u}. 
We refer to Appendix \ref{sec:regularity} for a complete proof.  
Further, given the decomposition $\frac{\d F}{\d p}(p,x) = g(x) + G(p,x)$ in \Cref{assum:decomposition-F-derivative} and $u_0=v_0+w_0$ in \Cref{assum:initialdistribution}, we are tempted to decompose the solution to \cref{eq:fixpoint_m-to-u} as \(u = v + w\), where \(v\) solves the HJB equation corresponding to the convex part
\begin{equation}
\label{eq:HJB-strongly-convex}
\partial_t v_t = \frac{\sigma^2}{2} \Delta v_t - \frac{\sigma^2}{4} \left|\nabla v_t\right|^2 + g -\gamma v_t,
\end{equation}
and \(w\) solves the remaining part
\begin{equation}
\label{eq:HJB-MF-part}
\partial_t w_t = \frac{\sigma^2}{2} \Delta w_t - \frac{\sigma^2}{2} \nabla v_t \cdot \nabla w_t - \frac{\sigma^2}{4} \left|\nabla w_t\right|^2 + G\left(m_t, \cdot\right) - \gamma w_t.
\end{equation}
Because it is a special case of \cref{eq:fixpoint_m-to-u},
 \cref{eq:HJB-strongly-convex} also admits a unique classical solution, and therefore so does \cref{eq:HJB-MF-part}. 
The proof of Proposition \ref{prop:hjb-main} is completed through Propositions \ref{prop:V-convexity},  \ref{prop:W-Lipschitz} and \ref{lem:hessian-u-bounded} below.

\begin{rem}\label{rmk:u_to_v}
In case $G=0$ and $w_0 = 0,$ we have $u=v$. Therefore all the properties proved for the function $u$ are shared by the function $v$.
\end{rem}

\begin{lem}\label{lem:FBSDEshorthorizon}
Let $u$ be the classical solution to \cref{eq:fixpoint_m-to-u}. There exists a constant $\d>0$ only depending on $\ol\kappa,  \ol\eta_0, L_0, L_G$ from Assumption~\ref{assum:decomposition-F-derivative} and \ref{assum:initialdistribution}
such that $\sup_{T\le \delta} \|\nabla^2 u(T,\cdot)\|_\infty <\infty$.
\end{lem}

\begin{proof}
\noindent{\rm (i).}\quad We first show that the SDE \eqref{eq:forwardSDE} below admits a unique strong solution. Define $\psi(t,x): = e^{-\frac12 u(t,x)}.$ By \Cref{lem:feynmankac psi} in appendix, we have
\begin{equation}\label{eq:FeynmanKacLemma f}
\psi(t,x) = \dbE\left[e^{-\frac12 \int_0^t \big(\frac{\d F}{\d p}(m_{t-s}, x+\si W_s)- \gamma u(t-s, x+\si W_s)\big)ds} \psi_0(x+\si W_t)\right].
\end{equation}
Now consider the continuous paths space $C([0,T])$ as the canonical space. Denote by $\ol \dbF:=(\ol\cF_t)_{t\le T}$ the canonical filtration and $\ol X$ the canonical process. Let $\dbP$ be the probability measure such that $(\ol X-x)/\si$ is a $\dbP$--Brownian motion starting from the origin. 
We may define an equivalent probability measure $\dbQ$ on the canonical space via
\begin{equation}\label{eq:dQdPdensity}
 \frac{d \dbQ}{d \dbP}\Big|_{\ol\cF_T} = \Lambda_T
  := \frac{e^{-\int_0^T \frac12 \big( \frac{\d F}{\d p}(m_{T-s}, \ol X_s)- \gamma u(t-s, \ol X_s)\big) dt} \psi_0(\ol X_T)}{\psi(T,x)}.
\end{equation}
 By It\^o's formula, we may identify that
 \begin{align*}
 \dbE^\dbP\left[\Lambda_T \,|\, \ol\cF_t \right] & = \frac{e^{-\int_0^t \frac12 \big( \frac{\d F}{\d p}(m_{T-s}, \ol X_s)- \gamma u(t-s, \ol X_s)\big) dt} \psi(T-t, \ol X_t)}{\psi(T,x)} \\
  & = \exp\left(-\int_0^t \frac12 \nabla u(t-s, \ol X_s) \cdot d\ol X_s - \int_0^t \frac{\si^2}{8} |\nabla u(t-s, \ol X_s)|^2 ds \right).
\end{align*}
Using the Girsanov's theorem, we may conclude that the SDE
\begin{equation}\label{eq:forwardSDE}
 X_t =x -\int_0^t \frac{\si^2}{2}\nabla u(T- s, X_s )ds + \si W_t,
\end{equation}
admits a weak solution. In addition, since $x\mapsto \nabla u(t,x)$ is locally Lipschitz, the SDE above has the property of pathwise uniqueness. Therefore, we can conclude by Yamada--Watanabe's theorem.

\noindent {\rm (ii).}\quad Next we observe that $\nabla u$ is the classical solution to
\begin{equation}\label{eq:pde-nabla-u}
\partial_t \nabla u_t = \frac{\si^2}{2} \Delta \nabla u_t - \frac{\si^2}{2} \nabla^2 u_t \nabla u_t + \nabla \frac{\d F}{\d p}(m_t, \cdot) - \gamma \nabla u_t.
\end{equation}
By denoting
\( Y_t: = \nabla u(T-t, X_t),\)
it follows from It\^o's formula that $(X, Y)$ solves the forward-backward SDE (FBSDE for short):
\begin{align*}
\begin{cases}
d X_t = -\frac{\si^2}{2} Y_t dt + \si dW_t, & X_0 =x, \\
d Y_t = \big( \gamma Y_t - \nabla \frac{\d F}{\d p}(m_{T-t},  X_t)\big) dt + Z_t dW_t , & Y_T = \nabla u_0(X_T),
\end{cases}
\end{align*}
where $Z_t = \si\nabla^2 u(T-t, X_t)$. Introduce the norm
\[\|(Y, Z)\|_{\cD}: = \sup_{t\le T}\left\{ \dbE\left[ |Y_t|^2 + \int_t^T |Z_s|^2 ds
\right] \right\}^\frac12.\]
We are going to show that $\|(Y, Z)\|_{\cD}<\infty$, provided that $T$ is small enough.

By \Cref{lem:bound of u} and \Cref{prop:regularity-heat-equation} in appendix, we have
\[e^{-C_T(1+|x|^2)} \le \psi(t,x) \le C_T,\qquad |\nabla \psi(t,x)| \le C_T(1+|x|^2).\]
Therefore,
\[|\nabla u(t,x)| =2 \frac{|\nabla \psi|}{\psi}(t,x) \le C_T (1+|x|^2) e^{C_T|x|^2}. \]
On the other hand, by the definition of $\Lambda_T$ in \eqref{eq:dQdPdensity}, we have
\[\Lambda_T\le C_Te^{C_T(|x|^2 + \sup_{t\le T} |\ol X_t|^2)}.\]
Now we may provide the following estimate
\begin{align*}
 \dbE\left[\sup_{t\le T} |Y_t|^2\right]
 &= \dbE\left[\sup_{t\le T} |\nabla u(T-t, X_t)|^2\right]
 = \dbE^\dbP\left[ \Lambda_T \sup_{t\le T} |\nabla u(T-t, \ol X_t)|^2\right] \\
& \le C_T e^{C_T|x|^2} \dbE^\dbP\left[(1+\sup_{t\le T}|\ol X_t|^2 )e^{C_T \sup_{t\le T}|\ol X_t|^2}\right].
\end{align*}
In particular, if $T$ is small enough, we have $\dbE^\dbP\left[(1+\sup_{t\le T}|\ol X_t|^2 ) e^{C_T \sup_{t\le T}|\ol X_t|^2}\right]<\infty$.

 Moreover, by It\^o's formula, we obtain
\begin{align*}
 d |Y_t|^2
 &= \left(2\gamma |Y_t|^2-2Y_t\cdot\nabla \frac{\d F}{\d p}(m_{T-t},  X_t) + |Z_t|^2\right) dt + 2Y_t \cdot Z_t dW_t\\
 &\ge \left((2\gamma -1) |Y_t|^2 - |\nabla \frac{\d F}{\d p}(m_{T-t},  X_t)|^2 + |Z_t|^2\right) dt + 2Y_t \cdot Z_t dW_t.
\end{align*}
Define the stopping time $\t_n : = \inf\{t\ge0:~ |Z_t |\ge n\}$, and note that
\begin{multline*}
\dbE\left[\int_0^{T\wedge\t_n} |Z_t|^2 dt\right] \le \dbE[|Y_{T\wedge\t_n}|^2] -\dbE[|Y_0|^2] \\ 
+ \dbE\left[\int_0^{T\wedge\t_n} \Big( (1-2\gamma)|Y_t|^2 + |\nabla \frac{\d F}{\d p}(m_{T-t},  X_t)|^2 dt \Big)\right]. 
\end{multline*}
Since we have proved $ \dbE\left[\sup_{t\le T} |Y_t|^2\right]<\infty $, by monotone and dominated convergence theorem, we obtain
\[\dbE\left[\int_0^{T } |Z_t|^2 dt\right] \le \dbE[|Y_{T }|^2] + \dbE\left[\int_0^{T } \Big((1-2\gamma)|Y_t|^2 + |\nabla \frac{\d F}{\d p}(m_{T-t},  X_t)|^2 dt \Big)\right] <\infty. \]
Therefore, we have $\|(Y, Z)\|_{\cD} <\infty$.

\noindent {\rm (iii).}\quad It is known (see e.g. \cite[Theorem I.5.1]{MY99}) that there exists $\d>0$ only depending on $\ol\kappa,  \ol\eta_0, L_0, L_G$ such that for $T\le \d$ the process $(Y,Z)$ here is the unique solution to the FBSDE such that $\|(Y,Z)\|_{\cD}<\infty$. Moreover, by standard a priori estimate (again see \cite[Theorem I.5.1]{MY99}) we may find a constant $C\ge 0$ only depending on $\ol\kappa,  \ol\eta_0, L_0, L_G$ such that for $(Y', Z')$ solution to the FBSDE above starting from $X_0 =x'$ we have
\[\|(Y, Z)-(Y', Z')\|_{\cD} \le C |x-x'|, \quad\mbox{for}\quad T\le \d. \]
In particular,  it implies that
\begin{equation*}
 \left|Y_0-Y_0'\right| =\left|\nabla u(T,x) - \nabla u(T,x')\right|\leq C|x-x'|,
\end{equation*}
so that $\sup_{T\le \delta} \|\nabla^2 u(T,\cdot)\|_\infty <\infty.$
\end{proof}

\begin{prop}\label{prop:V-convexity}
Let $v$ be the classical solution to \cref{eq:HJB-strongly-convex}. It holds:
\begin{enumerate}

 \item[\rm(i)] The function $v_t$ is $\eta_t$--convex,  \ie, $\nabla^2 v_t \ge \eta_t I_d,$ with
\begin{equation}
\label{eq:theta-t-evolution}
\frac{d \eta_t}{d t} = \ul\kappa - \gamma\eta_t - \si^2 \eta_t^2, \qquad \eta_0 = \ul\eta_0.
\end{equation}
In particular,  $v_t$ is  $\ul{\eta}$--convex with $\ul{\eta} := \min(\eta_0,\frac{\sqrt{\gamma^2+ 4\si^2 \ul\kappa} -\gamma}{2\si^2}).$

\item[\rm (ii)] The Hessian of $v$ is bounded uniformly w.r.t. $t>0$ and $m.$
\end{enumerate}
\end{prop}

\begin{proof}
We divide the following discussion into $3$ steps.

 \noindent {\rm (i).} \quad We first prove the strict convexity of the solution $v$ on a short horizon. Fix $T:= \d$ small enough so that, thanks to \Cref{lem:FBSDEshorthorizon}, $\nabla^2 v$ is uniformly bounded on $(0,T]$. We shall prove that not only $\nabla^2 v$ has a positive lower bound, but also the bound does not depend on $T$.

 As in Step {\rm (i)} of the proof of \Cref{lem:FBSDEshorthorizon}, we may define the strong solution 
 \[X_t = x -\int_0^t \frac{\si^2}{2} \nabla v(T-s,X_s) ds +\si W_t.\] 
 Further define $Y_t :=\nabla v(T-t, X_t)$ and $Z_t:= \si \nabla^2 v(T-t, X_t)$ so that $\|(Y, Z)\|_{\cD}<\infty$ and $(Y,Z)$ is the unique solution to the FBSDE on the short horizon $[0,T]$:
 \begin{align*}
 \begin{cases}
 d X_t = -\frac{\si^2}{2} Y_t dt + \si dW_t, & X_0 =x, \\
 d Y_t = \big( \gamma Y_t - \nabla g( X_t) \big)dt + Z_t dW_t , & Y_T = \nabla v_0( X_T).
 \end{cases}
 \end{align*}
 Define $(X',Y', Z')$ similarly with $X'_0 =x'$, and further denote by $\d X_t := X_t - X'_t, ~ \d Y_t := Y_t - Y'_t, ~ \d Z_t := Z_t - Z'_t $. Note that due to the uniqueness of the solution to the FBSDE, we have $\d X_t = \d Y_t = \d Z_t = 0$ for $t\ge \tau:=\inf\{t\ge 0: ~ \d X_t =0\}$. By It\^o's formula, it is easy to verify that
 \begin{multline*}
 d \frac{\d X_t \cdot \d Y_t}{|\d X_t|^2}
 = \left( -\frac{\si^2|\d Y_t|^2 }{2{|\d X_t|^2}}
 +\gamma \frac{\d X_t\cdot \d Y_t}{|\d X_t|^2}
 - \frac{\d X_t \cdot \big(\nabla g(X_t) - \nabla g(X'_t)\big)}{|\d X_t|^2} + \si^2\frac{|\d X_t \cdot \d Y_t|^2 }{|\d X_t|^4} \right) dt\\
 + \frac{\d X_t \cdot \d Z_t dW_t}{|\d X_t|^2 }.
 \end{multline*}
 Therefore, the pair $(\hat Y_t, \hat Z_t):= \left(\frac{\d X_t \cdot \d Y_t}{|\d X_t|^2},\frac{\d X_t^\top \d Z_t }{|\d X_t|^2 } \right)$ solves the BSDE:
 \[ d \hat Y_t = \left( -\frac{\si^2|\d Y_t|^2 }{2{|\d X_t|^2}} +\gamma \hat Y_t - \frac{\d X_t \cdot \big(\nabla g(X_t) - \nabla g(X'_t)\big)}{|\d X_t|^2} +\si^2 \hat Y_t^2 \right)dt + \hat Z_t dW_t. \]
 According to \Cref{lem:FBSDEshorthorizon}, the process $\hat Y$ is bounded on $[0,T]$ and so is the coefficient in front of $dt$ above. By the It\^o isometry, we clearly have $\dbE[\int_0^T |\hat Z_t|^2 dt] <\infty$. 
 
 We aim at providing a lower bound for $\hat Y$. Consider the Riccati equation \eqref{eq:theta-t-evolution} and note that the solution $(\eta_t)_{t\ge 0}$ evolves monotonously from the initial condition $\eta_0>0$ to the positive equilibrium  $\eta^*:=(\sqrt{\gamma^2+ 4\si^2 \ul\kappa} -\gamma)/2\si^2.$ In particular, it holds
 \begin{equation}\label{eq:eta_bound}
   \ul\eta=\min(\eta_0,\eta^*) \leq \eta_t \le \max(\eta_0,\eta^*).
 \end{equation}
 Define $\hat \eta_t := \eta_{T-t} $ for $t\le T$ so that
 \[d \hat \eta_t = (-\ul \kappa + \gamma \hat \eta_t +\si^2 \hat\eta^2_t )dt, \quad \hat\eta_T \le \hat Y_T. \]
 Since $g$ is $\ul \kappa$--convex, we have
 \begin{align*}
 d (\hat Y_t -\hat\eta_t)
 &= \left( -\frac{\si^2|\d Y_t|^2 }{2{|\d X_t|^2}}- \frac{\d X_t \cdot \big(\nabla g(X_t) - \nabla g(X'_t)\big)}{|\d X_t|^2} + \ul \kappa + \gamma (\hat Y_t -\hat \eta_t) +\si^2 (\hat Y_t^2- \hat\eta_t^2) \right)dt + \hat Z_t dW_t\\
 &\le \left(\gamma (\hat Y_t -\hat \eta_t) + \si^2 (\hat Y_t+ \hat\eta_t) (\hat Y_t- \hat\eta_t) \right) dt + \hat Z_t dW_t
 \end{align*}
 Since $\hat Y_t, ~\hat\eta_t$ are both bounded and $\dbE[\int_0^T |\hat Z_t|^2 dt] <\infty$, it follows from the standard comparison principle for BSDE that $\hat Y_t - \hat \eta_t \ge 0$, \ie, the function $v_t$ is $\eta_t$--convex for $t\in [0,T]$.

\noindent {\rm (ii).}\quad We shall improve the bound of $|\nabla^2 v|$ to get a bound independent of the horizon $T = \d$. Note that $\nabla v$ satisfies the equation
\[\partial_t \nabla v_t = \frac{\sigma^2}{2} \Delta \nabla v_t - \frac{\sigma^2}{2} \nabla^2 v_t \nabla v_t + \nabla g -\gamma \nabla v_t.\]
Thus it admits the probabilistic representation
\begin{gather*}
 \nabla v (t,x) = \dbE\left[\int_0^t e^{-\gamma s}\nabla g(X_s) ds + e^{-\gamma t}\nabla v_0(X_t)\right],
 \intertext{with}
 X_s = x -\int_0^s\frac{\si^2}{2} \nabla v(t-r, X_r) dr + \si W_s.
\end{gather*}
Let $X'$ be the solution to the SDE above with $X'_0=x'.$ Since $\nabla g$ and $\nabla v_0$ are both Lipschitz continuous, we have
\begin{equation}\label{eq:diff nabla v}
|\nabla v (t,x) -\nabla v (t,x')|
\le \dbE\left[\bar{\kappa} \int_0^t |X_s - X'_s| ds + \bar{\eta}_0|X_t - X'_t|\right],
\end{equation}
Now recall that we have proved in Step {\rm (i)} that the function $v_s$ is $\eta_s$--convex for $s\in [0,t]$ so that 
\begin{align*}
 \frac 12 d \left|X_s - X'_s\right|^2 &= \left(X_s - X'_s \right) \cdot \left( dX_s - dX'_s \right) \\
 &= - \frac{\sigma^2}{2} \left(X_s - X'_s \right) \cdot \left(\nabla v \left(t-s, X_s\right) - \nabla v \left(t-s, X'_s\right)\right) ds \\
 &\leq - \frac{\sigma^2 \eta_{t-s}}{2} \left|X_s - X'_s\right|^2 ds,
\end{align*}
Furthermore recall that  $\eta_s\geq \ul\eta$ for all $s\geq 0$ by \cref{eq:eta_bound} so that
\begin{equation}\label{eq:flow-lipschitz}
 |X_s - X'_s|\le e^{-\frac{\si^2\ul\eta s}{2} }|x-x'|.
\end{equation}
Together with \cref{eq:diff nabla v}, we obtain
\[|\nabla v (t,x) -\nabla v (t,x')| \le C\left(1 + \frac{2}{\si^2 \ul\eta }\right)|x-x'|.\]
Therefore $|\nabla^2 v(t,\cdot)| \le C\left(1 + \frac{2}{\si^2 \ul\eta }\right)$, in particular the bound does not depend on $T=\d$.

\noindent {\rm (iii).}\quad By the result of Step {\rm (ii)}, we know that 
$\nabla^2 v(\d, \cdot)$ is bounded and the bound does not depend on $\d$. Together with \Cref{lem:FBSDEshorthorizon}, we conclude that $\nabla^2 v$ is bounded on $[\d, 2\d]$, and further deduce that $v_t$ is $\eta_t$--convex and $\nabla^2 v$ has a $\d$-independent bound again on $[\d, 2\d]$ thanks to the results of Step {\rm (i)--(ii)}. Therefore the desired result follows from induction.
\end{proof}

\begin{prop}
\label{prop:W-Lipschitz}
 Let $w$ be the classical solution to \cref{eq:HJB-MF-part}. 
 Then the function $x\mapsto w(t,x)$ is Lipschitz continuous uniformly w.r.t. $t \geq 0$ and $m.$
\end{prop}

\begin{proof}
We consider the following stochastic control problem. Let $(\Omega, \mathcal{F}, \dbP, \dbF)$ be a filtered probability space, and $W$ be a $(\dbP,\dbF)$-Brownian motion. Denote by $\cA$ the collection of admissible control process, \textit{i.e.}, $\alpha$ is progressively measurable and $\dbE\left[\int_0^t |\alpha_t|^2 dt\right]<\infty$.
Then it follows from standard dynamic programming arguments that 
\[
 w(t,x) = \inf_{\a\in\cA}\Expect \left[\int_0^t e^{-\gamma s}\left(G\left(m_{t-s}, X^{\alpha}_s\right) + \frac{\sigma^2}{4}|\alpha_s|^2\right) ds + e^{-\gamma t} w_0\left(X^{\alpha}_t\right)\right],
\]
where $X^{\alpha}$ stands for the strong solution to
\[dX^{\alpha}_s = - \frac{\sigma^2}{2} \left( \nabla v \left(t - s, X^{\alpha}_s \right) + \alpha_s \right)ds + \sigma dW_s, \quad X^{\alpha}_0 = x.
\]
Denote by $Y^{\alpha}$ the solution to the SDE above with $Y^{\alpha}_0=y.$ Then it holds
\begin{equation}\label{eq:diff valuefunctions}
 \left|w\left(t, y\right) - w\left(t, x\right)\right| \leq \sup_{\alpha} \Expect \left[ L_G \int_0^t e^{-\gamma s} \left| Y^{\alpha}_s - X^{\alpha}_s\right| ds + L_0 e^{-\gamma t} \left|Y^{\alpha}_t - X^{\alpha}_t\right|\right].
\end{equation}
Using the convexity of $ v_s$ from \Cref{prop:V-convexity}, we obtain by the same argument as~\cref{eq:flow-lipschitz} that 
\[ \left|Y^{\alpha}_s - X^{\alpha}_s \right| \leq e^{-\frac{\si^2\ul\eta s}{2} }\left|y-x \right|.\]
Together with \cref{eq:diff valuefunctions}, we can find a \((t, m)\)--independent constant \(L>0\) such that \[|w\left(t, y\right) - w\left(t, x\right)| \leq L \left|y - x\right|.\]
\end{proof}

 Given the decomposition of $u$ as the sum of a convex and a Lipschitz function,   we shall also prove that the Hessian of $u$ is bounded uniformly in time which is clearly an improvement over \Cref{lem:FBSDEshorthorizon}.

\begin{prop}\label{lem:hessian-u-bounded}
 Let $u$ be the classical solution to \cref{eq:fixpoint_m-to-u}.  Then the Hessian of $u$ is bounded uniformly w.r.t. $t>0$ and $m.$
\end{prop}

\begin{proof}
Recall that $\nabla u$ satisfies \cref{eq:pde-nabla-u} so that,
by Feynman--Kac's formula, it admits the probabilistic representation
\begin{gather}\label{eq:nabla u FeynmanKac}
  \nabla u(t,x) = \dbE\left[\int_0^t e^{-\gamma s} \nabla \frac{\delta F}{\delta p}(m_{t-s}, X_s)ds + e^{-\gamma t}\nabla u_0(X_t)\right], 
 \intertext{with} 
  X_s = x - \frac{\si^2}{2} \int_0^s \nabla u(t-r, X_r) dr + \si W_s.\notag
\end{gather}
Let us prove that $x\mapsto \nabla u(t,x)$ is Lipschitz continuous with a Lipschitz constant independent of $t$ and $m$. Denote by $Y$ the solution to the SDE above with $Y_0=y.$
It follows from  the reflection coupling \Cref{thm:reflectioncoupling} in appendix that
for $p^X_s :=\cL(X_s)$ and $p^Y_s := \cL(Y_s),$
\[ \cW_1(p^X_s, p^Y_s) \le C e^{- c \si^2 s} \cW_1(p^X_0, p^Y_0), \quad\mbox{for all $s\ge 0$}.\]
Note that the drift $\nabla u = \nabla v + \nabla w$ satisfies \Cref{assum:reflectioncoupling} since $v$ is $\ul\eta$--convex and $\nabla w$ is bounded, see \Cref{rem:decom-drift}.
Together with \cref{eq:nabla u FeynmanKac} and the fact that $\nabla \frac{\delta F}{\delta p}(p,\cdot)$ and $\nabla u_0$ are uniformly Lipschitz,  we have by Kantorovitch duality that
\[|\nabla u(t,x) - \nabla u(t,y)|
\le C \left(\int_0^t \cW_1(p^X_s , p^Y_s) ds + \cW_1(p^X_t, p^Y_t)\right)\le C |x-y|,\]
where the constant $C$ does not depend on $t$ and $m$.
\end{proof}

\subsection{Proof of Theorem \ref{thm:wellposedness}}
\label{sec:proof-wellposedness}

\begin{proof}[Proof of \Cref{thm:wellposedness}]
In view of Proposition~\ref{prop:hjb-main}, it is enough to show that the mapping~\eqref{eq:contraction} $(m_t)_{t\in [0,T]}\mapsto (p_t)_{t\in [0,T]}$ is a contraction for $T$ small enough, where $p_t = e^{-u_t}/\int {e^{-u_t}}$ with $u$ the solution to \cref{eq:fixpoint_m-to-u}. This contraction property relies essentially on a reflection coupling argument established in Appendix~\ref{sec:reflectioncoupling} which follows from the decomposition of $u$ as the sum of a convex and a Lipschtz function.

\noindent {\rm (i).}\quad Let $(\tilde{m}_t)_{t\in [0,T]}$ be another flow of probability measures satisfying \Cref{assum:mLipschitz}, and use it to define the function $\tilde{u}$ as in \cref{eq:fixpoint_m-to-u}. Denote by $\d u:= u -\tilde{u}$. Using the stability result for the HJB equation \eqref{eq:fixpoint_m-to-u} proved in \Cref{prop:HJB stability} below,  we obtain
\begin{equation}\label{eq:estimate nabla u by m}
 \sup_{t\le T} \|\nabla \d u(t, \cdot) \|_\infty \le T C_T\sup_{t\le T}\cW_1(m_t, \tilde{m}_t).
\end{equation}

\noindent {\rm (ii).}\quad Further define the probability density $\tilde{p}_t = e^{-\tilde{u}_t}/\int {e^{-\tilde{u}_t}}$. Note that $p_t$ and $\tilde{p}_t$ are the invariant measures of the diffusion processes
\begin{equation*}
 dX_s = -\nabla u(t, X_s)ds + \sqrt{2} dW_s, \qquad d\tilde{X}_s = -\nabla \tilde{u}(t, \tilde{X}_s)ds + \sqrt{2} dW_s,
\end{equation*}
respectively. Denote by $p_{t,s} := \cL(X_s)$ and $\tilde{p}_{t,s} : = \cL(\tilde{X}_s)$ the marginal distributions, and assume that $p_{t,0} = \tilde{p}_{t,0}=p_0$. 
By Proposition~\ref{prop:hjb-main} and \Cref{rem:decom-drift},  we may apply the reflection coupling in \Cref{thm:reflectioncoupling} in appendix to obtain
\[\cW_1\big(p_{t,s}, \tilde{p}_{t,s}\big) \le Ce^{-c s} \int_0^s e^{c r} \|\nabla \d u(t, \cdot)\|_\infty dr.\]
Let $s\rightarrow \infty$ on both sides. Since $\lim_{s\rightarrow\infty}\cW_1(p_{t,s}, p_t) = 0$ and $\lim_{s\rightarrow\infty}\cW_1(\tilde{p}_{t,s}, \tilde{p}_t) = 0$ by \Cref{rem:ergodicity}, we deduce that
\[\cW_1\big(p_t, \tilde{p}_t\big) \le C\|\nabla \d u(t, \cdot)\|_\infty. \]

\noindent {\rm (iii).}\quad Together with \cref{eq:estimate nabla u by m}, we finally obtain
\[\sup_{t\le T} \cW_1\big(p_t, \tilde{p}_t\big) \le TC_T \sup_{t\le T}\cW_1\big(m_t, \tilde{m}_t\big). \]
Therefore, given $T$ small enough, the mapping $(m_t)_{t\le T}\mapsto (p_t)_{t\le T}$ is a contraction under the metric $\sup_{t\le T} \cW_1(\cdot_t, \cdot_t)$.
\end{proof}

The following lemma shows that the gradient $\nabla u$ of the solution to the HJB equation \eqref{eq:fixpoint_m-to-u} is stable with respect to $(m_t)_{t\in [0,T]}$ as needed for the proof of Theorem~\ref{thm:wellposedness} above, as well as with respect to $\nabla u_0$  for later use.

\begin{lem}\label{prop:HJB stability}
Let $\tilde{u}$ be the classical solution to \cref{eq:fixpoint_m-to-u} corresponding to the flow of distribution $\tilde m$ satisfying Assumption~\ref{assum:mLipschitz} and the initial value $\tilde u_0$ satisfying \Cref{assum:initialdistribution}.
Then we have the following stability results:
\begin{itemize}
 \item[\rm (i)] If $\nabla u_0 =\nabla \tilde u_0$, then
 \( \|\nabla \d u(t, \cdot)\|_\infty \le C_t \int_0^t \cW_1(m_s, \tilde m_s)ds.\)

 \item[\rm (ii)] Otherwise 
  \(\|\nabla \d u (t,\cdot)\|_{(2)} \le C_t \left( \int_0^t \cW_1(m_s, \tilde m_s)ds + \|\nabla\d u_0\|_{(2)} \right).\)

\end{itemize}
\end{lem}

\begin{proof}
Similiar to~\cref{eq:nabla u FeynmanKac}, it follows from the Feynman-Kac's formula that
\begin{align*}
 & \nabla u(t,x) = \dbE\left[ \int_0^t e^{-\gamma s }\nabla \frac{\d F}{\d p}(m_{t-s}, X_s ) ds + e^{-\gamma t}\nabla u_0(X_t) \right], \\
 & \nabla \tilde u(t,x) = \dbE\left[ \int_0^t e^{-\gamma s } \nabla \frac{\d F}{\d p}(\tilde m_{t-s}, \tilde X_s ) ds + e^{-\gamma t } \nabla \tilde u_0(\tilde X_t) \right], 
\end{align*}
with 
\begin{align*}
 &dX_s = -\frac{\si^2}{2}\nabla u(t-s, X_s) ds + \si dW_s,\quad X_0 = x,\\
 &d\tilde X_s = -\frac{\si^2}{2}\nabla \tilde u(t - s, \tilde X_s) ds + \si dW_s,\quad \tilde X_0=x.
\end{align*}
By Proposition~\ref{prop:hjb-main} and \Cref{rem:decom-drift}, we may apply the reflection coupling in \Cref{thm:reflectioncoupling} in appendix to compare the marginal distribution of $X$ and $\tilde X$, denoted by $p$ and $\tilde p$ respectively. We obtain
\[ \cW_1(p_s, \tilde p_s) \le C e^{- c \si^2 s} \int_0^s e^{c \si^2 r}\dbE\big[| \nabla \d u (t-r,X_r)|\big] dr . \]
Further, by Kantorovich duality and Lipschitz continuity of $\nabla \frac{\d F}{\d p}$ and $\nabla \tilde u_0,$ we have
\begin{multline*}
 |\nabla \d u(t,x)|
 \le C \dbE\Big[ \int_0^t \int_0^s C e^{-\gamma s -c\si^2(s-r) }\big| \nabla \d u(t-r, X_r) \big| dr ds +\int_0^t e^{-\gamma s} \cW_1(m_{t-s}, \tilde m_{t-s})ds \\
  + \int_0^t C e^{-\gamma t -c\si^2(t-s)}|\nabla\d u(t-s, X_s)|ds + e^{-\gamma t }|\nabla\d u_0(X_t)|\Big],
\end{multline*}
which implies that 
\begin{equation}\label{eq:nabla u estimate}
 |\nabla \d u(t,x)|
 \le C \dbE\Big[ \int_0^t \big| \nabla \d u(t-s, X_s) \big| ds +\int_0^t \cW_1(m_{s}, \tilde m_{s})ds +|\nabla\d u_0(X_t)|\Big].
\end{equation}
Recall the decomposition of the solution established in Proposition~\ref{prop:hjb-main}:
\(u = v + w, \, \tilde u = \tilde v+ \tilde w,\)
where $v ,\tilde v$ are strictly convex and $w, \tilde w$ are Lipschitz.
We divide the following discussion into two cases.

\noindent{\rm (i).}\quad We assume $\nabla \d u_0 = 0$. Note that in this case $\nabla v = \nabla \tilde v$ (because $v,\tilde v$ are not influenced by $m$ or $\tilde m$) and that $\nabla \d u = \nabla w -\nabla \tilde w$ is bounded. It follows from the \cref{eq:nabla u estimate} that
\[ \|\nabla \d u(t, \cdot)\|_\infty \le C \left(\int_0^t \|\nabla \d u(s, \cdot )\|_\infty ds + \int_0^t \cW_1(m_s, \tilde m_s)ds\right). \]
Finally, by the Gr\"onwall inequality, we obtain
\[ \|\nabla \d u(t, \cdot)\|_\infty \le C_t \int_0^t \cW_1(m_s, \tilde m_s)ds . \]

\noindent{\rm (ii).}\quad We consider the general case. Recall that both $\nabla v$ and $\nabla \tilde v$ are Lipschitz, and both $\nabla w$ and $\nabla \tilde w$ are bounded, so we have $\|\nabla \d u (t,\cdot)\|_{(2)}  <\infty $. Further it follows from \cref{eq:nabla u estimate} that
\begin{align*}
 |\nabla \d u(t,x)|
 & \le
 C \Big( \int_0^t \|\nabla \d u (t-s,\cdot)\|_{(2)} (1+\dbE [|X_s|^2]) ds \notag\\
 & \quad\quad\quad\quad +\int_0^t \cW_1(m_s, \tilde m_s)ds + \|\nabla\d u_0\|_{(2)} (1+\dbE[|X_t|^2]) \Big) \notag\\
 & \le C e^{ct} \Big( \int_0^t \|\nabla \d u (t-s,\cdot)\|_{(2)} (1+|x|^2) ds \notag\\
 & \quad\quad\quad\quad +\int_0^t \cW_1(m_s, \tilde m_s)ds + \|\nabla\d u_0\|_{(2)} (1+|x|^2) \Big).
\end{align*}
Finally, by the Gr\"onwall inequality, we obtain
\[\|\nabla \d u (t,\cdot)\|_{(2)} \le C_t \left( \int_0^t \cW_1(m_s, \tilde m_s)ds + \|\nabla\d u_0\|_{(2)} \right). \]
\end{proof}

\subsection{Properties of Mean Field   Schr\"odinger Dynamics}

The decomposition of the generalized MFS dynamics provided by \Cref{thm:wellposedness} allows us to derive Gaussian bounds, first locally in time as stated below and later uniformly in time, see \Cref{lem:property-p-uniform}. 

\begin{prop}\label{lem:property-p}
For any $T>0,$ there exist $\ul c, \ol c, \ul C, \ol C>0,$ such that for all $t\in [0,T],$ $x\in\dbR^d,$
\begin{equation*}
  \ul C e^{- \ul c |x|^2} \le  p_t(x) \le \ol C e^{- \ol c |x|^2}. 
\end{equation*}
In particular, $p_t \in \cP_H$ for all $t\geq 0.$
\end{prop}

\begin{proof}
 The Gaussian bounds  follow immediately from \Cref{lem:property-p-app} in appendix, whose assumptions are satisfied on $\mathcal{T}=[0,T]$ according to \Cref{thm:wellposedness}.
 Then we observe
 \begin{equation*} 
  \left|\nabla \sqrt{p_t}\right|^2 = \frac{1}{4}  \left|\nabla \log p_t\right|^2 p_t \leq C_T(1+|x|^2) p_t,
 \end{equation*}
 where the latter follows from the boundedness of $\nabla^2 \log p_t.$ Thus $\nabla \sqrt{p_t}\in L^2$ and $p_t\in\cP_H.$
\end{proof}

Then we establish a stability result for the generalized MFS dynamics \eqref{eq:gradient-flow}.  It plays a crucial role in the proof of convergence in \Cref{thm:convergence}.

\begin{prop}\label{prop:stabailityinitialvalue}
For $n\in\dbN$, let $p^n$ (resp. $p$) be the generalized MFS dynamics \eqref{eq:gradient-flow} starting from $p^n_0$ (resp. $p_0$), where $p^n_0$ (resp. $p_0$) satisfy \Cref{assum:initialdistribution}. If $ \nabla \log p^n_0$ converges to $ \nabla \log p_0$ in $ \|\cdot\|_{(2)}$, then $(p^n_t, \log p^n_t)$ converges to $(p_t, \nabla \log p_t)$ in $\cW_1\otimes \|\cdot\|_{(2)}$ for all $t>0$.
\end{prop}

\begin{proof}
Recall that the function $u_t$ solution to \eqref{eq:HJB} differs from $-\log p_t$ only through an additive constant (depending on $t$), in particular $\nabla u_t = -\nabla \log p_t.$
Denote by $\d u := u^n - u$.
By the stability result of the HJB equation \eqref{eq:fixpoint_m-to-u} proved in \Cref{prop:HJB stability}, we have
\begin{equation}\label{eq:stability estimate in proof}
 \|\nabla \d u(T, \cdot)\|_{(2)} \le C_T \left( \int_0^T \cW_1(p^n_t, p_t)dt + \|\nabla\d u_0\|_{(2)} \right).
\end{equation}
As in the proof of \Cref{thm:wellposedness}, note that $p^n_t$ and $p_t$ are the invariant measures of the diffusions:
\begin{equation*}
 dX^n_s = -\nabla u^n(t, X^n_s )ds + \sqrt{2}dW_s,\qquad 
 dX_s = -\nabla u(t, X_s )ds + \sqrt{2}dW_s,
\end{equation*}
respectively. Denote the marginal distributions $p^n_{t,s} := \cL(X^n_s)$ and $p_{t,s} : = \cL(X_s)$, and assume that $p^n_{t,0} = p_{t,0}$. Using the reflection coupling, we deduce from \Cref{thm:reflectioncoupling} that
\[\cW_1(p^n_{t,s}, p_{t,s}) \le C e^{- cs} \int_0^s e^{c r}\dbE\big[|\nabla \d u(t, X_r)|\big] dr.\]
By letting $s\rightarrow\infty$ on both sides, it follows from using successively the $\cW_1$--convergence of $p^n_{t,s}$ and $p_{t,s}$ toward $p^n_{t}$ and $p_{t}$ by \Cref{rem:ergodicity}, the linear growth of $\nabla\d u(t,\cdot)$ and  \Cref{lem:property-p} that
\[\cW_1(p^n_t, p_t) \le C \int_{\dbR^d}|\nabla \d u(t, x)| p_t (x) dx
\le C_T \|\nabla \d u(t, \cdot)\|_{(2)}. \]
Together with \cref{eq:stability estimate in proof}, by the Gr\"onwall inequality, we obtain
\begin{gather*}
 \|\nabla \d u(T, \cdot)\|_{(2)} \le C_T e^{T C_T } \|\nabla \d u_0\|_{(2)},
 \intertext{as well as}
\cW_1(p^n_T, p_T) \le C_T e^{T C_T } \|\nabla \d u_0\|_{(2)}. 
\end{gather*}
\end{proof}

\section{Convergence towards the Minimizer}
\label{sec:mfconvergence}

\subsection{First Order Condition}

The aim of this section is to derive a first order condition to characterize the minimizer of the generalized free energy $\cF^{\sigma,\gamma}.$ 
Recall that $\cF^{\sigma,\gamma}(p) = F(p) + \si^2 I (p) + \gamma H(p)$ with parameters $\sigma>0,$ $\gamma\geq 0,$ and $I(p) = \int |\nabla \sqrt p|^2,$  $H(p)= \int p\log p.$

\begin{prop}
\label{thm:convexity}
The function $\fF^{\si, \gamma}$ is convex on $\cP_H.$ Additionally, if it admits a minimizer \(p^* \in \mathcal P_H\) such that \(1/p^* \in L_\text{loc}^\infty\), then it is unique.
\end{prop}

\begin{proof}
It follows  from the convexity of $F$ by \Cref{assum:potential},  the convexity of $H$ by convexity of $x\mapsto x\log(x)$ and \Cref{prop:convexity-fisher} below.
\end{proof}

\begin{lem}
\label{prop:convexity-fisher}
Let \(p, q\in\cP_H\) and \(\alpha, \beta > 0\). Then we have
\[
 I\left(\alpha p + \beta q\right) \leq \alpha I\left(p\right) + \beta I\left(q\right).
\]
If in addition \(1/p \in L^\infty_\text{loc}\), then the equality holds if and only if \(p=q\).
\end{lem}

\begin{proof}
Let \(\varphi = \sqrt p, \psi = \sqrt q\). We have by using the Cauchy--Schwarz inequality
\begin{align*}
 I \left(\alpha p + \beta q \right) & = \int \left| \nabla \sqrt{\alpha \varphi^2 + \beta \psi^2}\right|^2 = \int \frac{\left( \alpha \varphi \nabla \varphi + \beta \psi \nabla \psi \right)^2}{\alpha \varphi^2 + \beta \psi^2} \\
& \leq \int \frac{\left( \alpha \varphi^2 + \beta \psi^2\right) \left(\alpha \left(\nabla \varphi \right)^2 + \beta \left(\nabla \psi \right)^2\right)}{\alpha \varphi^2 + \beta \psi^2} = \alpha I\left(p\right) + \beta I\left(q\right).
\end{align*}
The equality holds if and only if \( \varphi \nabla \psi = \psi \nabla \varphi\). 
If in addition \(\frac{1}{p} \in L^\infty_\text{loc}\) then \(\frac{1}{\varphi} \in L^\infty_\text{loc}\) and \(\frac{\psi}{\varphi} \in L^1_\text{loc}\) which is a distribution in sense of Schwartz. Its derivative satisfies
\[
\nabla \left(\frac{\psi}{\varphi}\right) = \frac{\varphi \nabla \psi - \psi \nabla \varphi}{\varphi^2} = 0.
\]
Therefore \(\frac{\psi}{\varphi}\) is constant a.e., \ie, $p$ and $q$ are proportional.
\end{proof}

\begin{prop}
\label{thm:foc}
If a probability measure \(p \in \mathcal P_H\) satisfies $p\in C^2$ and
\begin{equation*}
 p(x) \le C e^{- c |x|^2}, \qquad \left|\nabla^2 \log p(x)\right| \le C, \label{eq:hyp-foc-entropy}
\end{equation*}
then the following inequality holds: for all $q\in \cP_H,$
\begin{equation*}
 \fF^{\sigma, \gamma}(q) - \fF^{\sigma, \gamma}(p) \ge \int_{\dbR^d} \frac{\delta \fF^{\sigma, \gamma}}{\delta p}(p,x) \left(q(x)-p(x)\right)\,dx.
\end{equation*}
In particular, if $\frac{\delta \fF^{\sigma, \gamma}}{\delta p}(p,\cdot)=0,$ then $p$ is the unique minimizer of the generalized free energy~$\fF^{\si, \gamma}.$
\end{prop}

\begin{proof}
We have $\fF^{\sigma, \gamma}(p)= F(p) + \sigma^2 I (p) + \gamma H(p).$ We deal with each of these three terms separately.  Adding the three subsequent inequalities gives the desired inequality. The second assertion then follows immediately from \Cref{thm:convexity}. Throughout the proof,  we denote \(p_t := p + t (q-p)\) for \(t \in \left[0,1\right]\).

\noindent{\rm (i).}\quad By convexity of $F,$ it holds
\begin{equation*}
 F(q) - F(p) \ge \frac{F(p_t) - F(p)}{t}.
\end{equation*}
Since $F$ is $\cC^1,$ we conclude by passing to the limit $t\to 0$ that
\begin{equation*}
 F(q) - F(p) \ge  \left.\frac{ d F \left(p_t\right)}{dt} \right|_{t=0^+} = \int_{\dbR^d} \frac{\delta F}{\delta p}(p,\cdot) \left(q-p\right).
\end{equation*}

\noindent{\rm (ii).}\quad Denote \(I_K \left(p\right) := \frac 14 \int_K \frac{\left|\nabla p\right|^2}{p}\) for \(K \subset \mathbb R^d\) compact and \(p\in\cP_H\). 
Assume first that $q$ is bounded and compactly supported. 
Then it follows from the convexity of $I_K$ and differentiation under the integral sign that
\[
  I_K \left(q\right) - I_K\left(p\right) \geq \left.\frac{ d I_K \left(p_t\right)}{dt} \right|_{t=0^+} = - \frac 14\int_K \frac{\left| \nabla p \right|^2}{p^2} (q-p) + \frac 12 \int_K \frac{\nabla p \cdot \nabla (q-p)}{p}.
\]
Note that $\nabla q = 2 \sqrt{q} \nabla \sqrt{q}\in  L^2.$
Next we take the limit $K\uparrow \dbR^d$ and we observe that the r.h.s. converges by using for the first term, $\frac{\left| \nabla p(x) \right|}{p(x)} = \left| \nabla \log p(x) \right|\leq C(1+|x|)$ and $p,q\in\cP_2,$ and for the second term, \(\frac{\left|\nabla p\right|^2}{p} = 4\left|\nabla\sqrt{p}\right|^2 \in L^1\).
Using further integration by parts, since $p(x) \le C e^{- c |x|^2}$ and $q$ is compactly supported, we obtain 
\begin{equation*}
 I \left(q\right) - I\left(p\right) \geq  - \frac 14 \int_{\dbR^d} { \left( \frac{\left|\nabla p\right|^2}{{p}^2} + 2 \nabla \cdot \left(\frac{\nabla p}{p}\right) \right) (q-p)}.
\end{equation*}
To conclude it remains to deal with the general case $q\in\cP_H$ not necessarily bounded and compactly supported. Given $M>0,$ we consider the distribution $q_M\propto\mathbf{1}_{|x|\le M} q\wedge M$ and we apply the inequality above to $q_M.$ Taking the limit $M\to\infty$ yields the desired result as the r.h.s.  converges since $q\in\cP_ 2$ and $|\nabla^2 \log p | \le C.$

\noindent{\rm (iii).}\quad 
Denote \(H_K \left(p\right) := \int_K p \log p \) for \(K \subset \mathbb R^d\) compact and \(p\in\cP_H\). 
Assume first that $q$ is bounded. Then it follows from the convexity of $H_K$ and differentiation under the integral sign that  
\begin{equation*}
 H_K \left(q\right) - H_K\left(p\right) \geq  \left.\frac{ d H_K \left(p_t\right)}{dt} \right|_{t=0^+} = \int_K {\left(1 + \log p\right)\left(q-p\right)}.
\end{equation*}
Next we take the limit $K\uparrow \dbR^d$ and we observe that the r.h.s. converges as $p,q\in\cP_2$ and $\left| \log p(x) \right|\leq C(1+|x|^2)$.  We obtain 
\begin{equation*}
 H \left(q\right) - H\left(p\right) \geq  \int_{\dbR^d} { \log{p} \,(q-p)}
\end{equation*}
To conclude it remains to deal with the general case $q\in\cP_H$ not necessarily bounded. Given $M>0,$ we consider the distribution $q_M\propto q\wedge M$ and we apply the inequality above to $q_M \in L^{\infty}.$ Taking the limit $M\to\infty$ yields the desired result.
\end{proof}

\subsection{Dissipation of Energy}

\begin{prop}\label{thm:energydecrease}
The generalized free energy decreases along the generalized MFS dynamics $(p_t)_{t\ge 0}$ solution to~\cref{eq:gradient-flow}. More precisely, we have
\begin{equation}\label{eq:energy_decrease}
 \frac{d}{dt} \fF^{\si, \gamma} (p_t) = - \int_{\dbR^d} \left|\frac{\d \fF^{\si, \gamma}}{ \d p}(p_t, x)\right|^2 p_t(x) dx.
\end{equation}
\end{prop}

\begin{proof}
Using \Cref{thm:foc} whose assumptions are satisfied in view of \Cref{thm:wellposedness} and \Cref{lem:property-p}, we have
\begin{align*}
 \fF^{\sigma, \gamma}(p_{t+h}) - \fF^{\sigma, \gamma}(p_t)
 &\ge \int_{\dbR^d} \frac{\d \fF^{\si, \gamma} }{\d p}(p_t, x) (p_{t+h}-p_t)(x) dx \\
 & = - \int_{\dbR^d} \frac{\d \fF^{\si, \gamma} }{\d p}(p_t, x) \int_t^{t+h}\frac{\d \fF^{\si, \gamma} }{\d p}(p_s, x) p_s(x) ds dx
\end{align*}
Similarily we have
\[ \fF^{\sigma, \gamma} (p_{t+h}) - \fF^{\sigma, \gamma}(p_t)
\le -\int_{\dbR^d} \frac{\d \fF^{\si, \gamma} }{\d p}(p_{t+h}, x) \int_t^{t+h}\frac{\d \fF^{\si,\gamma} }{\d p}(p_s, x) p_s(x) ds dx. \]
The conclusion then follows from the dominated convergence theorem.  Indeed,  by \Cref{thm:wellposedness},
the mapping $t\mapsto \frac{\d \fF^{\si,\gamma} }{\d p}(p_t, x)$ is continuous and satisfies $\sup_{t\le T} \left|\frac{\d \fF^{\si, \gamma}}{\d p}(p_t,x ) \right| \le C_T (1+ |x|^2)$ for any $T>0.$ 
Note that the same holds for $\frac{\d F}{\d p}(p_t,x )$ by the $\cW_1$--continuity of $t\mapsto p_t.$
In addition,  \Cref{lem:property-p} ensures that $ \int |x|^{4}\sup_{t \le T} p_t(x)dx <\infty.$
\end{proof}

The dissipation of  energy allows us to extend previous estimates of the generalized MFS dynamics from $[0,T]$ to $[0,\infty)$ which is crucial to study its asymptotic behavior.

\begin{lem}\label{lem:energy bound for p}
  It holds 
 \begin{equation}\label{eq:energy bound for p}
  \sup_{t> 0}\left\{\int_{\dbR^d} |x|^2 p_t(x) dx + \int_{\dbR^d} |\nabla\sqrt{p_t}(x)|^2 dx \right\}<+\infty.
 \end{equation}
\end{lem}

\begin{proof}
Let  $q$ the Gaussian density with variance $\upsilon^2.$ 
We have
\begin{equation*}
 H(p) = H(p\,|\,q) + \int { p(x) \log q(x) \,dx} \geq -\frac{d}{2}\log(2\pi\upsilon^2) - \frac{1}{2\upsilon^2} \int_{\dbR^d} |x|^2 p(x) dx 
\end{equation*}
Then it follows from Assumption~\ref{assum:potential} by choosing $\upsilon$ sufficiently large that  there exist $C,c>0$ such that
\begin{equation}\label{eq:free-energy-below}
 \fF^{\sigma, \gamma}(p_t) \ge -C+c \int_{\dbR^d} |x|^2 p_t(x) dx + \sigma^2 \int_{\dbR^d} |\nabla\sqrt{p_t}(x)|^2 dx,\qquad\forall\,t\ge 0.
\end{equation}
Since the generalized free energy is decreasing according to \Cref{thm:energydecrease}, we deduce that 
 \begin{equation*}
  \sup_{t\ge 0}\left\{c \int_{\dbR^d} |x|^2 p_t(x) dx + \sigma^2 \int_{\dbR^d} |\nabla\sqrt{p_t}(x)|^2 dx \right\} 
  \le  C + \fF^{\sigma, \gamma}(p_0).
 \end{equation*}  
\end{proof}

\begin{prop}\label{lem:nabla u uniform bound}
 It holds for all $x\in\dbR^d,$ 
 \[\sup_{t\ge 0}|\nabla \log p_t (x)| \le C(1+|x|).\]
\end{prop}

\begin{proof}
In view of \Cref{thm:wellposedness}, the Hessian $\nabla^2 \log p_t$ is bounded by some constant, denoted~$L.$ 
In  particular, it holds 
\[|\nabla \log p_t (x)| \le L|x| + |\nabla \log p_t(0)|,\]
and also 
\[|\nabla \log p_t(0)|^2 \le (L|x| + |\nabla \log p_t (x)|)^2\le 2L^2|x|^2 + 2|\nabla \log p_t (x)|^2.\]
It follows that
\[4 \int_{\dbR^d} |\nabla \sqrt{p_t} (x)|^2 dx
= \int_{\dbR^d} |\nabla \log p_t (x)|^2 p_t(x) dx
\ge \frac12 |\nabla \log p_t(0)|^2 - L^2\int_{\dbR^d} |x|^2 p_t(x)dx.\]
We conclude by \Cref{lem:energy bound for p} that $\sup_{t\ge 0}|\nabla \log p_t(0)|<\infty.$
\end{proof}

 Using \Cref{lem:nabla u uniform bound}, it is straightforward to extend the Gaussian bounds of \Cref{lem:property-p} from $[0,T]$ to $\dbR_+.$

\begin{cor}\label{lem:property-p-uniform}
There exist $\ul c, \ol c, \ul C, \ol C >0,$ such that for all $t\geq 0,$ $x\in\dbR^d,$
\begin{equation*}
  \ul C e^{- \ul c |x|^2} \le  p_t(x) \le \ol C e^{- \ol c |x|^2}.
\end{equation*}
\end{cor}

\subsection{Proof of Theorem \ref{thm:convergence}}
\label{sec:convergence}

\begin{proof}[Proof of \Cref{thm:convergence}]
We start by observing that the family $(p_t)_{t\ge 0}$ is relatively compact for the uniform norm on $C(\dbR^d).$ This property follows from Arzel\`a--Ascoli Theorem as
\begin{equation}\label{eq:ascoli}
 p_{t}(x) \le C e^{-c|x|^2}, \qquad |\nabla p_{t}(x)| = |\nabla \log p_t(x)| p_{t}(x) \leq C(1+|x|) e^{-c|x|^2},
\end{equation}
by \Cref{lem:nabla u uniform bound} and \Cref{lem:property-p-uniform}.  Let $p^*$ be an arbitrary cluster point, \ie, $p_{t_k}$ converges uniformly to $p^*$ for some sequence $t_k\uparrow\infty.$ Note that, in view of the Gaussian bound above, the convergence also occurs in $\cW_p$ for any $p\ge 1.$   The aim of the proof is to show that $p^*$ is the unique minimizer of $\fF^{\si, \gamma}.$

\noindent{\rm (i).}\quad Let us show first that, for almost all $h>0,$ 
\begin{equation}\label{eq:lasalle}
\liminf_{k\rightarrow \infty} \int_{\dbR^d} \left|\frac{\d \fF^{\sigma, \gamma}}{\d p}(p_{t_k+h}, x)\right|^2 p(t_k+h,x)dx = 0. 
\end{equation}
Indeed, suppose by contradiction that there exists $h>0$ such that
\begin{align*}
0 &< \int_0^h\liminf_{k\rightarrow\infty} \left\{ \int_{\dbR^d} \left|\frac{\d \fF^{\sigma, \gamma}}{\d p}(p_{t_k+s}, x)\right|^2 p_{t_k+s}(x)dx \right\}ds \\
&\le \liminf_{k\rightarrow\infty} \int_0^h \left\{ \int_{\dbR^d} \left|\frac{\d \fF^{\sigma, \gamma}}{\d p}(p_{t_k+s}, x)\right|^2 p_{t_k+s}(x)dx \right\}ds,
\end{align*}
where the last inequality is due to Fatou's lemma. 
It would lead to a contradiction as by \Cref{thm:energydecrease},
\begin{align*} 
 \fF^{\sigma, \gamma}(p_{t_{k+1}}) - \fF^{\sigma, \gamma}(p_{t_0}) & = \sum_{j=0}^k {\fF^{\sigma, \gamma}(p_{t_{j+1}}) - \fF^{\sigma, \gamma}(p_{t_j})} \\
 & = - \sum_{j=0}^k {\int_0^{t_{j+1}-t_j} \int_{\dbR^d} \left|\frac{\d \fF^{\sigma, \gamma}}{\d p}(p_{t_j+s}, x)\right|^2 p_{t_j+s}(x)dx ds}
\end{align*} 
where the l.h.s. is bounded from below by \cref{eq:free-energy-below} and the r.h.s. diverges to $-\infty$ by assuming w.l.o.g. that $t_{j+1}-t_j \ge h.$

\noindent{\rm (ii).}\quad From now on, denote by $t_k^h:=t_k+h$ where $h>0$ is chosen so that \cref{eq:lasalle} holds. Let $q$ be an arbitrary probability measure in $\cP_H.$  Due to the first order inequality established in \Cref{thm:foc}, we have
\[\fF^{\sigma, \gamma} (q)- \fF^{\sigma, \gamma}(p_{t_k^h})
~\ge~ \int_{\dbR^d} \frac{\d \fF^{\si, \gamma}}{\d p}(p_{t_k^h} ,x) (q - p_{t_k^h})(x)dx.\]
In view of \Cref{thm:wellposedness},  \Cref{lem:nabla u uniform bound} and \Cref{lem:property-p-uniform}, we have 
\begin{equation*}
\sup_{t\ge 0} \left|\frac{\d \fF^{\si, \gamma}}{\d p}(p_t,x ) \right| \le C (1+ |x|^2), \qquad  \sup_{t\ge 0} \int_{\dbR^d} |x|^{2}p_t(x)dx <\infty.
\end{equation*}
Note that the first inequality holds for $\frac{\d F}{\d p}(p_t,x )$ since $(p_t)_{t\ge 0}$ belongs to a $\cW^1$--compact set  due to the Gaussian bound.
Hence, for any $\e>0,$ we can find $K$ big enough such that for all $k,j\in\dbN,$
\begin{equation*}
 \fF^{\sigma, \gamma} (p_{t_k^h}) \le \fF^{\sigma, \gamma} (q) -\int_{|x|\le K} \frac{\d \fF^{\si, \gamma}}{\d p}(p_{t_k^h} ,x) (q - p_{t_k^h})(x)dx + \e.
\end{equation*}
Further it follows from Cauchy--Schwartz inequality that
\[\left|\int_{|x|\le K} \frac{\d \fF^{\si, \gamma}}{\d p}(p_{t_k^h} ,x) (q - p_{t_k^h})(x)dx \right|
\le \left( \int_{\dbR^d} \left|\frac{\d \fF^{\sigma, \gamma}}{\d p}(p_{t_k^h}, x)\right|^2 p_{t_k^h}(x)dx \int_{|x|\le K}\frac{|q-p_{t_k^h}|^2}{p_{t_k^h}}(x) dx\right)^\frac12.\]
Assume first that $q$ is bounded and note that the second term on the r.h.s.  is also bounded as $\inf_{k,h,x} p_{t_k^h}(x)>0$ by \Cref{lem:property-p-uniform}. Thus we deduce by taking the limit $k\to \infty$ and then $\varepsilon\to 0$ that 
\begin{equation}\label{eq:minimum}
\liminf_{k\rightarrow\infty} \fF^{\si, \gamma} (p_{t_k^h})\le \fF^{\si, \gamma} (q),
\end{equation}
for any $q\in\cP_H$ bounded. If $q\in\cP_H$ is not necessarily bounded, this inequality also holds as it holds for the distribution $q_M\propto q\wedge M$ and $\fF^{\si, \gamma} (q_M)\to \fF^{\si, \gamma} (q)$ as $M\to\infty.$

\noindent{\rm (iii).}\quad Denote by $(p^*_t)_{t\ge 0}$ the solution to  \cref{eq:gradient-flow} starting from $p^*_0=p^*. $  
We observe by \Cref{lem:Lassalle} below that  $p_{t_k^h}$ and $\nabla \log p_{t_k^h}$ converges pointwise  to $p^*_{h}$ and $\nabla\log p^*_{h}$ respectively.
In view of \Cref{lem:nabla u uniform bound} and \Cref{lem:property-p-uniform}, it follows easily by the dominated convergence theorem that $\lim_{k\rightarrow\infty} F(p_{t_k^h})=F(p^*_h)$ as $p_{t_k^h}\to p^*_{h}$ in $\cW_2$ by using the Gaussian bound,
\begin{gather*}
\lim_{k\rightarrow\infty} H(p_{t_k^h}) = \lim_{k\rightarrow\infty} \int {p_{t_k^h} \log p_{t_k^h}} = \int {p^*_{h} \log p^*_{h}}=H(p^*_{h}),
 \intertext{and}
 \lim_{k\rightarrow\infty} I(p_{t_k^h}) = \lim_{k\rightarrow\infty} \frac 14 \int {|\nabla\log p_{t_k^h} |^2 p_{t_k^h}} = \frac 14 \int {|\nabla\log p^*_{h}|^2 p^*_{h}}=I(p^*_{h}).
\end{gather*}
We deduce that 
\begin{equation*}
 \lim_{k\rightarrow\infty} \fF^{\si, \gamma} (p_{t_k^h}) = \fF^{\si, \gamma} (p^*_{h}).
\end{equation*}
Hence,  by \cref{eq:minimum}, $p^*_{h}$ is a minimizer of $\fF^{\si, \gamma}.$ 
In view of \Cref{thm:convexity}, this minimizer is unique and thus $p^*_{h}$ does not depend on $h$ and coincides with its limit $p^*_0=p^*$ when $h\to 0.$

\noindent{\rm (iv).}\quad As a byproduct, we observe that $p^*$ is a stationary solution to \cref{eq:gradient-flow} and thus it satisfies \[\frac{\d \fF^{\si, \gamma}}{ \d p}(p^*, \cdot) = 0.\]

\end{proof}

\begin{lem}\label{lem:Lassalle}
Using the notations above,   as $k\to \infty,$ $p_{t_k^h}$ converges uniformly to $p^*_h$ and $\nabla \log p_{t_k^h}$ converges to $\nabla \log p^*_{h}$ in $\|\cdot\|_{(2)}.$
\end{lem}

\begin{proof}
\noindent{\rm (i).}\quad Let us show first that $\nabla \log p_{t_k}$ converges to $\nabla \log p^*$ in $\|\cdot\|_{(2)}$.
According to \Cref{thm:wellposedness} and \Cref{lem:nabla u uniform bound}, $(\nabla \log p_{t_k})_{k\in\dbN}$ lives in a $\|\cdot\|_{(2)}$-compact set of the form
\begin{equation*}
 \cK := \left\{f:\dbR^d \to \dbR;\ f \text{ is $C$--Lipschitz and }|f(0)|\le C\right\},
\end{equation*}
for some constant $C>0.$
Consequently, there is a subsequence and a function $f\in\cK$ such that $\lim_{k\rightarrow \infty }\|\nabla \log p_{t_k} - f\|_{(2)} =0 $. 
Therefore, we have for almost all $x,y\in\dbR^d,$
\begin{align*}
 \log p^* (x) - \log p^*(y)
 &= \lim_{k\rightarrow\infty} \Big(\log p_{t_k}(x) - \log p_{t_k}(y) \Big)\\
 &=  \lim_{k\rightarrow\infty} \int_0^1 \nabla \log p_{t_k} (s x +(1-s) y ) \cdot (x-y) ds\\
 &= \int_0^1 f ( s x +(1-s) y ) \cdot (x-y) ds.
\end{align*}
So $f = \nabla \log p^*$ and the desired result follows.

\noindent{\rm (ii).}\quad In view of \Cref{prop:stabailityinitialvalue}, it follows immediately from Step (i) that $(p_{t_k^h},\nabla \log p_{t_k^h})$ converges to $(p^*_{h},\nabla \log p^*_{h})$ in $\cW_1 \otimes \|\cdot\|_{(2)}.$ 
 It remains to prove that $p_{t_k^h}$ converges uniformly to $p^*_{h}.$ This is an easy consequence of Arzel\`a--Ascoli Theorem by \cref{eq:ascoli}.
\end{proof}

\subsection{Proof of Theorem \ref{thm:exp_convergence}}
\label{sec:exp-cvg}

The proof relies on the following functional inequality which is new to the best of our knowledge and may carry independent interest.

\begin{thm}\label{thm:functional_ineq}
Let $p(dx) = e^{-u(x)}dx$ satisfy a Poincar\'e inequality with constant $C_P$, i.e., for all \(f \in H^1(p)\) such that \(\int f dp = 0,\)
\begin{equation}
\label{eq:poincare}
\int f^2 dp \leq C_P \int |\nabla f|^2 dp.
\end{equation}
Assume that $u$ is weakly differentiable with $\nabla u\in L^2$ and define the operator
\(\cL := \Delta - \nabla u \cdot \nabla.  \)
Then we have for all $f\in W^{2,2}(p)$ such that $\cL f \in L^2(p),$
\begin{multline}\label{eq:functional_ineq}
    C_P^{-1} \left(\int_{\dbR^d} f(x) p(dx)\right)^2 \int_{\dbR^d} |\nabla f(x)|^2 p(dx)\\
    \le \int_{\dbR^d} f(x)^2  p(dx) \int_{\dbR^d} \big(\cL f(x)\big)^2  p(dx) - \left(\int_{\dbR^d} f(x) \cL f(x) p(dx) \right)^2.
\end{multline}
\end{thm}

\begin{rem}
Note that it follows from integration by parts that 
\begin{equation}\label{eq:ipp}
\int_{\dbR^d} \cL f(x) p(dx) = 0,\quad \int_{\dbR^d} |\nabla f(x)|^2 p(dx) = - \int_{\dbR^d} f(x) \cL f(x) p(dx).
\end{equation}
Moreover, if $p^f(dx) = f(x)^2 p(dx)$ is a probability measure then the right hand side of the inequality \eqref{eq:functional_ineq} is equal to the variance of $\frac{\cL f}{f}$ under $p^f$, namely, ${\rm Var}_{p^f}(\frac{\cL f}{f})$.
\end{rem}

\begin{proof}[Proof of \Cref{thm:functional_ineq}]
Let \(f = f_0 + \bar f\), where \(\bar f = \int f dp\) is the mean. For the right-hand side of the inequality \eqref{eq:functional_ineq},  we obtain by using successively $\int f_0 dp=0,$ $\int \cL f dp=0$ and Cauchy--Schwartz inequality, 
\begin{multline*}
\int_{\dbR^d} f^2  dp \int_{\dbR^d} \big(\cL f \big)^2  dp - \left(\int_{\dbR^d} f \cL f dp \right)^2\\
=  {\bar f}^2 \int (\mathcal Lf)^2 dp + \int f_0^2 dp \int (\mathcal Lf)^2 dp
- \left( \int f_0 \mathcal Lf dp \right)^2
~\geq~ \bar f^2 \int (\mathcal Lf)^2 dp.
\end{multline*}
Meanwhile for the left-hand side,  we obtain by \cref{eq:ipp}, Cauchy--Schwarz inequality and Poincar\'e inequality, 
\begin{multline*}
   \int |\nabla f|^2 dp
= -\int f \mathcal Lf dp
= -\int f_0 \mathcal Lf dp\\
\leq  \left( \int f_0^2 dp\right)^{1/2} \left(\int (\mathcal Lf)^2 dp\right)^{1/2}
\leq  C_\text{P}^{1/2} \left(\int |\nabla f|^2 dp \int (\mathcal Lf)^2 dp \right)^{1/2}.
\end{multline*}
The desired inequality follows by combining the estimates above.
\end{proof}

\begin{prop}
\label{prop:uniform-spectral-gap}
If \(u : \mathbb R^d \to \mathbb R\) decomposes as \(u = v + w\)
with \(v,w\in C^2\), \(\nabla^2 v \geq \eta I_d\) and \(|\nabla w| \leq L\),
then there exists a constant \(C_P = C(\eta, L, d)\)
such that the Poincar\'e inequality~\eqref{eq:poincare} holds.
\end{prop}

\begin{proof}
This is a direct consequence of Corollary 1.6 (1) in \cite{bakry2008simple}.
\end{proof}

\begin{proof}[Proof of \Cref{thm:exp_convergence}]
Recall that \(p_t\) is the classical solution to the MFS dynamics~\eqref{eq:mfSchr\"odinger}.
For each \(t > 0,\) denote \(F_t := \frac{\delta F}{\delta p}(p_t,\cdot)\) and define
\begin{equation}\label{eq:def_hatp}
    \hat p_t := \argmin_{p\in\cP_H} \left\{\int F_t d p + \frac{\sigma^2}{4} I(p)\right\}.
\end{equation}
We recognize that it is the minimizer of the mean field optimization problem if we replace
\(F(p)\) by \(\int F_t dp\).
According to \Cref{thm:convergence}, the minimizer $\hat p_t = e^{-\hat{u}_t}$ satisfies
\(\hat u_t = \hat v_t + \hat w_t\)
with \(\nabla^2 \hat v_t \geq \ul\eta I_d\) and \(|\nabla \hat w_t| \leq L\) for all \(t > 0\).
Thus \(\hat p_t\) verifies a Poincar\'e inequality with a constant
\(C_P\) independent of time by \Cref{prop:uniform-spectral-gap}.  
Note also that
\begin{equation}\label{eq:hat_u_t}
 \frac{\sigma^2}{2}\Delta \hat u_t - \frac{\sigma^2}{4} |\nabla \hat u_t|^2 + F_t - \hat \lambda_t = 0,
\end{equation}
where, by integration by parts,
\begin{equation}\label{eq:ipp2}
  \hat \lambda_t = \int { \left(\frac{\sigma^2}{2}\Delta \hat u_t - \frac{\sigma^2}{4} |\nabla \hat u_t|^2 + F_t \right) d \hat p_t} = \int { \left(\frac{\sigma^2}{4} |\nabla \hat u_t|^2 + F_t \right) d \hat p_t}.
\end{equation}

The desired result follows by applying the functional inequality \eqref{eq:functional_ineq} with distribution $\hat p_t$ and function $f_t = \sqrt{p_t/\hat p_t}.$ Let \(\mathcal L_t = \Delta - \nabla \hat u_t \cdot \nabla\)  and  observe by direct computation using $f_t = \exp((\hat u_t - u_t)/2)$ that 
\begin{equation*}
\frac{\mathcal L_t f_t}{f_t}
= \frac 12 \Delta \hat u_t - \frac 14 |\nabla \hat u_t|^2
- \left(\frac 12 \Delta u_t - \frac 14 |\nabla  u_t|^2\right).
\end{equation*}
Then it follows from \cref{eq:hat_u_t} that 
\begin{equation}\label{eq:Lf}
\frac{\mathcal L_t f_t}{f_t}
= \sigma^{-2} \hat \lambda_t
- \sigma^{-2}\left(\frac {\sigma^2}2 \Delta u_t - \frac {\sigma^2}4 |\nabla  u_t|^2 + F_t\right).
\end{equation}
Thus, by using \Cref{thm:energydecrease}, the right-hand side of \cref{eq:functional_ineq} corresponds to
\begin{align*}
\frac{d\fF^\sigma(p_t)}{dt}
& = - \int \left| \frac{\sigma^2}{2} \Delta u_t - \frac{\sigma^2}{4} |\nabla u_t|^2 + F_t - \lambda_t\right|^2 d p_t \\
& = - \int \left| \frac{\sigma^2}{2} \Delta u_t - \frac{\sigma^2}{4} |\nabla u_t|^2 + F_t - \hat \lambda_t\right|^2 d p_t + \left(\hat \lambda_t - \lambda_t\right)^2 \\
& = -\si^4 {\rm Var}_{p_t}\left(\frac{\mathcal L_t f_t}{f_t} \right),
\end{align*}
where, by integration by parts,
\begin{equation}\label{eq:ipp3}
  \lambda_t = \int { \left(\frac{\sigma^2}{2}\Delta u_t - \frac{\sigma^2}{4} |\nabla u_t|^2 + F_t \right) d p_t} = \int { \left(\frac{\sigma^2}{4} |\nabla u_t|^2 + F_t \right) d p_t}.
\end{equation}

As for the left-hand side of \cref{eq:functional_ineq}, we have for the first term
\begin{equation*}
 \int f_t d\hat p_t = \int \sqrt{p_t\hat p_t} dx \ge C >0,
\end{equation*}
by using the Gaussian bounds provided in \Cref{lem:property-p-uniform}.  
Regarding the second term, it holds by using \cref{eq:ipp} and \cref{eq:Lf},
\begin{equation*}
 \int |\nabla f_t|^2 d\hat p_t = - \int f_t \cL_t f_t  d \hat p_t = \sigma^{-2} \int  \left(\frac {\sigma^2}2 \Delta u_t - \frac {\sigma^2}4 |\nabla  u_t|^2 + F_t \right) dp_t -  \sigma^{-2}\hat\lambda_t.
\end{equation*}
Using further \cref{eq:ipp2} and \cref{eq:ipp3}, we obtain
\begin{align*}
 \sigma^{2}\int |\nabla f_t|^2 d\hat p_t & =  \int  \left(\frac {\sigma^2}4 |\nabla  u_t|^2 + F_t\right) dp_t -\int  \left( \frac{\sigma^2}4 |\nabla  \hat u_t|^2 + F_t\right) d\hat p_t \\
 & = \int F_t(d p_t - d\hat p_t) + \frac{\sigma^2}{4}(I(p_t) - I(\hat p_t)) \\
 & \ge \int F_t (dp_t - dp^*) + \frac{\sigma^2}{4}(I(p_t) - I( p^*)),
\end{align*}
where the last inequality follows from the optimality of $\hat p_t$ in \cref{eq:def_hatp}.

By \Cref{thm:functional_ineq} and the above computations, we deduce that 
\begin{align*}
    \frac{d\fF^\sigma(p_t)}{dt} 
    & \le -\frac{(C \si)^2}{C_P}\left(\int F_t (dp_t - dp^*) + \frac{\sigma^2}{4}(I(p_t) - I( p^*))\right)\\
    &\le -\frac{(C \si)^2}{C_P} \left(\fF^\si(p_t)- \fF^\si(p^*)  \right),
\end{align*} 
where  the last inequality is due to \Cref{thm:foc}.
Therefore, the exponential convergence of the free energy \eqref{eq:exp_convergence_energy} follows with 
 a constant \(c = \frac{(C \si)^2}{C_P} \).
 
 In order to obtain the exponential convergence of the relative Fisher information,  define $f^*_t : = \sqrt{p_t / p^*}$, $\cL^*:= \Delta - \nabla u^* \cdot\nabla $, and repeat the previous computation:
 \begin{align*}
     I(p_t|p^*)
     &= 4\int |\nabla f^*_t|^2 d p^* 
     = -4 \int f^*_t \cL^*f^*_t d p^* \\
    & = 4\si^{-2}\left(\int \frac{\d F}{\d p}(p^*,\cdot)(dp_t - dp^*) + \frac{\sigma^2}{4}(I(p_t) - I(p^*))\right)\\
    & \le 4\si^{-2} \left(\fF^\si(p_t)- \fF^\si(p^*)  \right).
 \end{align*}
\end{proof}

\section{Gradient Flow with Relative Entropy}
\label{sec:gradientflow}

Let $p^h_i$  be defined in \cref{eq:discrete GD}. The proof of Theorem~\ref{thm:gradientflow} essentially relies on applying Arzel\`a--Ascoli Theorem to the family $((t,x)\mapsto p^h_{\lfloor t/h\rfloor}(x))_{h>0}.$ To this end, we need to ensure equicontinuity and boundedness in the two subsequent sections.
In the sequel,  we fix a time horizon $T<\infty$ and we denote by $N_h:=\lfloor T/h\rfloor$.  

\subsection{Equicontinuity in Space}

The goal of this section is to obtain uniform Gaussian bounds for the family $(p^h_i)_{h, i\le N_h}$ as in \Cref{lem:property-p} and to deduce equicontinuity in space of the discrete flow.

\begin{prop}\label{cor:p properties}
For some $\ul C, \ol C, \ul c, \ol c>0,$ we have for all $h>0, i\le N_h, x\in\dbR^d,$ 
\[\ul C e^{-\ul c |x|^2} \le p^h_i(x) \le \ol C e^{-\ol c |x|^2}.\]
In addition,  it holds 
\begin{equation*}
 \sup_{h,i\leq N_h} \|\nabla p_i^h\|_{\infty} < +\infty.
\end{equation*}
\end{prop}

\begin{proof}
The Gaussian bounds are a  direct consequence of \Cref{lem:property-p-app}, whose assumptions are satisfied according to Lemmas \ref{lem:v^h_i uniform convex}--\ref{lem:nabla u^h_i(0) bound} below.  As for the second part, it follows from the identity $\nabla p_i^h = p_i^h \nabla \log p^h_i$ by using the Gaussian upperbound above and the fact that $|\nabla \log p^h_i(x)|\leq C(1+|x|)$ according to Lemmas \ref{lem: nabla^2 u^i_h bounded} and \ref{lem:nabla u^h_i(0) bound} below.
\end{proof}

Recall that the mapping $p^h_i$ is a solution to the stationary MFS equation~\cref{eq:p^h_i}. 
In other words, if we denote $u^h_i:=-\log(p_i^h),$ it holds
\begin{equation}\label{eq:u^h_i}
 \frac{\sigma^2}{2} \Delta u^{h}_i - \frac{\sigma^2}{4} \left|\nabla u^{h}_i\right|^2 + \frac{\delta F}{\delta p} \left(p^{h}_i, \cdot\right) +h^{-1} u^h_{i-1} - h^{-1} u^{h}_i  = \lambda^h_i,
\end{equation}
with
\begin{equation}\label{eq:lambda^h_i}
\lambda^h_i = \int_{\dbR^d} \left( \frac{\delta F}{\delta p} \left(p^h_i, \cdot \right) + h^{-1} \left(u^h_{i-1} - u^h_{i}\right) + \frac{\sigma^2}{2} \Delta u^h_i -  \frac{\sigma^2}{4} \left|\nabla u^h_i \right|^2 \right) p^h_i.
\end{equation}
The key point is to observe that we have the decomposition $u^h_i = v^h_i + w^h_i$ with $v^h_i$ uniformly convex and $w^h_i$ uniformly Lipschitz. It comes from using arguments similar to Section~\ref{sec:decomposition}.
In this setting there is a slight ambiguity in the definition of $v^h_i$ (and thus $w^h_i$) due to the normalizing constant $\lambda_i^h$. 
Let us define $v^h_i$ as the solution to
\begin{equation*}
 \frac{\sigma^2}{2} \Delta v^{h}_{i} - \frac{\sigma^2}{4} \left|\nabla v^{h}_{i}\right|^2 + g +h^{-1} v^h_{i-1} - h^{-1} v^{h}_{i} = 0.
\end{equation*}

\begin{lem}\label{lem:v^h_i uniform convex}
The function $(v^h_i)_{h,i\le N_h}$ are uniformly $\eta$--convex for some $\eta>0$.
\end{lem}

\begin{proof}
Observe that $v^h_i$ corresponds to the stationary solution to \cref{eq:HJB-strongly-convex} with parameter $\gamma=h^{-1}$ and convex term $g +h^{-1} v^h_{i-1}$ instead of $g.$
Due to \Cref{prop:V-convexity}, $v^h_{i}$ is $\eta^h_{i}$--convex with
\begin{align*}
 \eta^h_{i} &= \frac{\sqrt{h^{-2} + 4 \si^2 \Big(\ul\kappa+h^{-1} \eta^h_{i-1}\Big)} - h^{-1}}{2\si^2}\\
 & \ge  \frac{\sqrt{h^{-2} + 4 \si^2 \Big(\ul\kappa+h^{-1} \min \big(\eta^h_{i-1}, \sqrt{\ul\kappa}/\si\big)\Big)} - h^{-1}}{2\si^2}\\
 & \ge \min \big(\eta^h_{i-1}, \sqrt{\ul\kappa}/\si\big).
\end{align*}
Recall that $\eta^h_0= \ul\eta_0$. Finally we obtain that $v^h_{i}$ is $\min \big(\ul\eta_0, \sqrt{\ul\kappa}/\si\big)$--convex.
\end{proof}

\begin{lem}\label{lem:nabla^2 v^h_i uniform bound}
The Hessian's $(\nabla^2 v^h_i)_{h, i\le N_h}$ are uniformly bounded.
\end{lem}
\begin{proof}
As in \Cref{prop:V-convexity}, we may obtain the following probabilistic representation:
\begin{gather*}
  \nabla v^h_{i} (x) = \dbE\left[\int_0^t e^{- \frac{s}{h}}\Big(\nabla g(X_s) +h^{-1} \nabla v^h_{i-1} (X_s)\Big) ds + e^{- \frac{t}{h}}\nabla v^h_{i}(X_t)\right],
  \intertext{with}
 X_s = x -\int_0^s\frac{\si^2}{2} \nabla v^h_{i}(X_r) dr + \si W_s.
\end{gather*}
Let $X'$ satisfy the same SDE with initial value $x'$. Since $v^h_{i}$ is $\eta$--convex, it follows from the same arguments as \cref{eq:flow-lipschitz} that
\[\left|X_t -X'_t\right| \le e^{-\frac{\si^2 \eta t}{2}} \left|x-x'\right|.\]
Further we obtain
\begin{align*}
 \left|\nabla v^h_{i} (x) -\nabla v^h_{i} (x') \right|
 & \le \dbE\left[\int_0^t e^{- \frac{s}{h}}(\ol\kappa + h^{-1}\|\nabla^2 v^h_{i-1} \|_\infty) \left|X_s - X'_s\right| ds + e^{- \frac{t}{h}}\|\nabla^2 v^h_{i} \|_\infty \left|X_t - X'_t\right|\right]\\
 & \le \Big(\int_0^t e^{- (\frac{1}{h} + \frac{\si^2 \eta }{2}) s}(\ol\kappa + h^{-1}\|\nabla^2 v^h_{i-1} \|_\infty) ds + e^{-(\frac{1}{h} + \frac{\si^2 \eta }{2})t} \|\nabla^2 v^h_{i} \|_\infty\Big) \left|x-x'\right|.
\end{align*}
Letting $t\rightarrow \infty,$ we get
\[\|\nabla^2 v^h_{i} \|_\infty \le \frac{\ol\kappa h + \|\nabla^2 v^h_{i-1} \|_\infty }{ 1 + \frac{\si^2 \eta h}{2}}.\]
Therefore, we deduce by induction that
\[\|\nabla^2 v^h_{i} \|_\infty
~\le~
\frac{2\ol\kappa }{\si^2 \eta} \bigg(1- \frac{1}{\left( 1+ \frac{\si^2 \eta h}{2}\right)^{i}}\bigg) + \frac{\ol \eta_0 }{\left( 1+ \frac{\si^2 \eta h}{2}\right)^{i} }
~\le~
\frac{2\ol\kappa }{\si^2 \eta} + \ol\eta_0. \]
\end{proof}

\begin{lem}\label{lem:nabla w^h_i uniform bound}
The gradients $(\nabla w^h_i)_{h, i\le N_h}$ are uniformly bounded.
\end{lem}

\begin{proof}
Observe that $w^h_i = u^h_i - v^h_i$ satisfies
\begin{equation*}
\frac{\sigma^2}{2} \Delta w^h_i - \frac{\sigma^2}{2} \nabla v^h_i \cdot \nabla w^h_i - \frac{\sigma^2}{4} \left|\nabla w^h_i\right|^2 + G\left(p^h_i, \cdot\right) + h^{-1} w^h_{i-1} - h^{-1} w^h_i = \lambda^h_i.
\end{equation*}
As in \Cref{prop:W-Lipschitz}, we observe that $w^h_{i}$ is the value function of the following stochastic control problem:
\begin{equation*}
 w^h_{i}(x) = \inf_{\a}\Expect \Big[\int_0^t e^{- \frac{s}{h} }\left(G\left(p^h_{i}, X^{\alpha}_s\right) + h^{-1}w^h_{i-1}( X^{\alpha}_s) + \frac{\sigma^2}{4}|\alpha_s|^2 - \lambda^h_i\right) ds + e^{- \frac{t}{h}} w^h_{i}\left( X^{\alpha}_t\right)\Big],
\end{equation*}
with
\[dX^{\alpha}_s = - \frac{\sigma^2}{2} \left( \nabla v^h_{i} \left( X^{\alpha}_s \right) + \alpha_s \right)ds + \sigma dW_s, \quad X^{\alpha}_0 = x.
\]
Further as in \cref{eq:diff valuefunctions}, we may estimate
\begin{equation*}
 \left|w^h_{i}(x) -w^h_{i}(x') \right|
\le \Big(\int_0^t e^{-(\frac{1}{h} + \frac{\si^2 \eta }{2})s}(L_G + h^{-1} \|\nabla w^h_{i-1}\|_\infty)
+ e^{-(\frac{1}{h} + \frac{\si^2 \eta }{2})t} \|\nabla w^h_{i}\|_\infty \Big)|x-x'|.
\end{equation*}
Letting $T\rightarrow \infty,$ we obtain
\[\|\nabla w^h_{i} \|_\infty \le \frac{L_G h + \|\nabla w^h_{i-1} \|_\infty }{ 1 + \frac{\si^2 \eta h}{2}}.\]
Therefore, we deduce by induction that
\[\|\nabla w^h_{i} \|_\infty
~\le~
\frac{2L_G}{\si^2 \eta} +L_0 . \]
\end{proof}

\begin{lem}\label{lem: nabla^2 u^i_h bounded}
The Hessians $(\nabla^2 u^h_i)_{h, i\le N_h}$ are uniformly bounded.
\end{lem}
\begin{proof}
As in the proof of \Cref{lem:hessian-u-bounded}, the Feynman--Kac formula ensures that
\begin{gather*}
  \nabla u^h_{i}(x)
= \dbE\left[\int_0^\infty e^{-\frac{t}{h}}\left(\nabla \frac{\d F}{\d p}(p^h_{i}, X_t) + h^{-1} \nabla u^h_{i-1}(X_t)\right)dt\right],
\intertext{with}
  X_t =x -\frac{\si^2}{2}\int_0^t\nabla u^h_{i}(X_s)ds+\si W_t.
\end{gather*}
Let $Y$ satisfy the same SDE starting from $y$. By the reflection coupling in \Cref{thm:reflectioncoupling}, it holds
\[\cW_1\left(p^X_t, p^Y_t\right) \le Ce^{-ct} |x-y|, \]
where $p^X$ and $p^Y$ are the marginal distribution of $X$ and $Y$ respectively.
Then it follows by Kantorovitch duality that
\begin{align*}
 \big|\nabla u^h_{i}(x)-\nabla u^h_{i}(y)\big|
 &\le \int_0^\infty Ce^{-\frac{t}{h}-ct} |x-y| + \int_0^\infty e^{-\frac{t}{h}} h^{-1} \dbE\left[\big|\nabla u^h_{i-1}(X_t)-\nabla u^h_{i-1}(Y_t)\big| \right]\\
 & = \frac{Ch}{1+ch}|x-y| + \int_0^\infty e^{-\frac{t}{h}} h^{-1} \dbE\left[\big|\nabla u^h_{i-1}(X_t)-\nabla u^h_{i-1}(Y_t)\big| \right]dt .
\end{align*}
Next apply the same estimate on $\big|\nabla u^h_{i-1}(X_t)-\nabla u^h_{i-1}(Y_t)\big|$, and obtain
\begin{multline*}
 \big|\nabla u^h_{i}(x)-\nabla u^h_{i}(y)\big|
 \le \frac{2Ch}{1+ch}|x-y|\\
  + \int_0^\infty e^{-\frac{t_1}{h}} h^{-1} \int_0^\infty e^{-\frac{t_2}{h}} h^{-1} \dbE\left[\big|\nabla u^h_{i-2}\big(X^{(1)}_{t_1+t_2}\big)-\nabla u^h_{i-2}\big(Y^{(1)}_{t_1+t_2}\big)\big|\right]dt_2 dt_1,
\end{multline*}
with
\begin{equation*}
 X^{(1)}_0=x, \quad dX^{(1)}_t =
 \begin{cases}
 -\frac{\si^2}{2} \nabla u^h_{i}\big(X^{(1)}_t \big)dt + \si dW_t, & \mbox{for $t\in [0, t_1)$}\\
 -\frac{\si^2}{2} \nabla u^h_{i-1}\big(X^{(1)}_t \big)dt + \si dW_t, &\mbox{for $t\ge t_1$}.
 \end{cases}  
\end{equation*}
By repeating the procedure, we eventually obtain for $i\ge 1$
\begin{multline*}
 \big|\nabla u^h_{i}(x)-\nabla u^h_{i}(y)\big|
 \le \frac{Ch i}{1+ch}|x-y|\\
  + \int_0^\infty \cdots \int_0^\infty e^{-\frac{1}{h}\sum_{j=1}^{i} t_j } h^{-i} \dbE\left[\Big|\nabla u_0\Big(X^{(i-1)}_{\sum_{j=1}^i t_j}\Big)-\nabla u_0\Big(Y^{(i-1)}_{\sum_{j=1}^i t_j}\Big)\Big|\right]dt_i \cdots dt_1,
\end{multline*}
with
\begin{equation*}
 X^{(i-1)}_0=x, ~~ dX^{(i-1)}_t = -\frac{\sigma^2}{2}\nabla u^h_{j}\big(X^{(i-1)}_t \big)dt + \si dW_t,~~\mbox{for $t\in[t_{i-j}, t_{i+1-j})$}.
\end{equation*}  
Again it follows from the reflection coupling that
\[ \cW_1\left(p^{X^{(i-1)}}_t, p^{Y^{(i-1)}}_t \right) \le Ce^{-ct}|x-y|,\]
where $p^{X^{(i-1)}}, ~p^{Y^{(i-1)}}$ are the marginal distribution of $X^{(i-1)}, ~Y^{(i-1)}$ respectively. In particular, the constants $c, ~C$ do not depend on $(t_1, \cdots, t_{i-1})$ by Lemmas~\ref{lem:v^h_i uniform convex}--\ref{lem:nabla w^h_i uniform bound}. Finally we get
\begin{align*}
 \big|\nabla u^h_{i}(x)-\nabla u^h_{i}(y)\big|
 &\le \frac{Chi}{1+ch}|x-y|+ C \int_0^\infty \cdots \int_0^\infty e^{-(\frac{1}{h}+c) \sum_{j=1}^i t_j} h^{-i} |x-y| dt_i \cdots dt_1\\
 &\le C(T+1)|x-y|,
\end{align*}
and the desired result follows.
\end{proof}

\begin{lem}\label{lem:nabla u^h_i(0) bound}
The vectors $\big(\nabla u^h_i(0)\big)_{h, i\le N_h}$ are uniformly bounded.
\end{lem}

\begin{proof}
The proof follows similar arguments as \Cref{lem:energy bound for p} and \Cref{lem:nabla u uniform bound}.
First we observe that the sequence $\fF^\si( p^h_{i})$ is non--increasing as
\[ \fF^\si( p^h_{i}) \le \fF^\si( p^h_{i}) + h^{-1} H( p^h_{i}| p^h_{i-1}) \le \fF^\si( p^h_{i-1}) + h^{-1} H( p^h_{i-1}| p^h_{i-1}) =  \fF^\si( p^h_{i-1}),\]
by using \cref{eq:discrete GD} for the second inequality.
In addition, it follows from Assumption~\ref{assum:potential} that 
\begin{equation*}
  \lambda \int_{\dbR^d}|x|^2 p^h_i(x) dx + \si^2 \int_{\dbR^d} |\nabla \sqrt{p^h_i}(x)|^2 dx 
\le \fF^\si( p^h_i) .
\end{equation*}
Therefore we have
\begin{equation*}
 \sup_{h, i\le N_h}\left\{ \lambda \int_{\dbR^d}|x|^2 p^h_i(x) dx + \si^2 \int_{\dbR^d} |\nabla \sqrt{p^h_i}(x)|^2 dx \right\}
\le \fF^\si( p_0) .
\end{equation*}
Since we have proved that $L:=\sup_{h, i\le N_h} \|\nabla^2 u^h_i\|_\infty <\infty$, we deduce that
\[4 \int_{\dbR^d} |\nabla \sqrt{p^h_i} (x)|^2 dx
= \int_{\dbR^d}  |\nabla u^h_i(x)|^2 p^h_i (x) dx
\ge \frac12 |\nabla u^h_i(0)|^2 - L^2 \int_{\dbR^d} |x|^2 p^h_i(x) dx.\]
Finally we obtain $\sup_{h, i\le N_h}|\nabla u^h_i (0)|<\infty $.
\end{proof}

\subsection{Equicontinuity in Time}

We aim to show the equicontinuity in time of the family  $(p^h)_{h>0}$ as stated in the proposition below. We also demonstrate as a preliminary step and for later use that the family of function $(t\mapsto \lambda^h_{\lfloor t/h\rfloor})_{h>0}$ defined by \eqref{eq:lambda^h_i} is bounded and equicontinuous.

\begin{prop}
\label{lem:compactness-lambda^h_i}
There exists constants \(C,c>0\) such that for all $h>0,$ $i<j\leq N_h,$ $x\in\dbR^d,$
\[
 |p^h_j(x) - p^h_i(x)| \leq C \exp(- c|x|^2) (j - i) h.
\]
Additionally, the sequence \((\lambda_i^h)_{h.i\leq N_h}\) is uniformly bounded,\ie,
\(
\sup_{h,i\leq N_h} |\lambda_i^h | < +\infty,
\)
and there exists a modulus of continuity \(\varpi : \mathbb R_+ \to \mathbb R_+\) such that for all $h>0,$ $ i<j\leq N_h,$
\[
 |\lambda_j^h - \lambda_i^h|
\leq \varpi ((j-i)h).
\]
\end{prop}

\begin{proof}
\noindent{\rm (i).} \emph{Formulas for \(\lambda_i^h\).}
The normalization condition for \(u_{i}^h, i \le N_h,\) writes
\begin{align*}
1 = \int \exp(- u_{i}^h) & = \int \exp (- u_{i-1}^h) \exp \left(- h \frac{u_{i}^h - u_{i-1}^h}{h} \right) \\
& = \int p_{i-1}^h \exp \left( -h \left( \frac{\sigma^2}{2} \Delta u_{i}^h - \frac{\sigma^2}{4} |\nabla u_{i}^h|^2 + \frac{\delta F}{\delta p} (p_{i}^h, \cdot) - \lambda_{i}^h \right) \right).
\end{align*}
where the latter follows from \cref{eq:u^h_i}. 
This allows us to obtain the following formula for \(\lambda_{i}^h\):
\begin{equation}
\label{eq:lambda_i+1^h}
\lambda_{i}^h = - \frac 1h \log \int p_{i-1}^h \exp(-h B_{i}^h),
\end{equation}
where 
\begin{equation*}
 B_{i}^h := \frac{\sigma^2}{2} \Delta u_{i}^h - \frac{\sigma^2}{4} |\nabla u_{i}^h|^2 + \frac{\delta F}{\delta p} (p_{i}^h, \cdot).
\end{equation*}
By writing the normalization in the backward way,
\[
1 = \int \exp(-u_{i-1}^h) = \int \exp(-u_{i}^h) \exp\left(h\frac{u_{i}^h - u_{i-1}^h}{h}\right) = \int \exp(-u_{i}^h) \exp( h (B_{i}^h - \lambda_{i}^h)),
\]
we obtain a similar formula
\begin{equation}
\label{eq:lambda_i+1^h-2}
\lambda_{i}^h = \frac 1h \log \int p_{i}^h \exp(hB_{i}^h).
\end{equation}
We apply Jensen's inequality to \cref{eq:lambda_i+1^h} and \cref{eq:lambda_i+1^h-2} to obtain
\begin{equation}
\label{eq:sandwich-lambda_k^h}
\int p_{i}^h B_{i}^h \leq \lambda_{i}^h \leq \int p_{i-1}^h B_{i}^h.
\end{equation}
Additionally, estimates from \Cref{lem: nabla^2 u^i_h bounded} and \Cref{lem:nabla u^h_i(0) bound} gives us the bound
\begin{equation}
\label{eq:bound-B_k^h}
 \sup_{h,i\leq N_h} |B_i^h(x)|\le C(1 + |x|^2).
\end{equation}
Note that the same holds for $\frac{\delta F}{\delta p} (p_{i}^h, \cdot)$ as $(p_{i}^h)_{h, i\le N_h}$ belong to a $\cW_1$--compact set due to the Gaussian bound.
Thus, by \Cref{cor:p properties}, we prove the second claim \(\sup_{h,i\leq N_h} |\lambda_i^h| < +\infty\).

\noindent{\rm (ii).} \emph{Time regularity of \(p_i^h\).} According to the HJB equation \eqref{eq:u^h_i} and Step (i) above, it holds
\begin{equation}\label{eq:uih_time_regularity}
| u_j^h(x) - u_i^h(x) | = \left| h \left(\sum_{s=i+1}^j B_s^h - \sum_{s=i+1}^j \lambda_s^h \right)\right|
\leq C (j - i) h (1 + |x|^2).
\end{equation}
Using further the bound from \Cref{cor:p properties}, we obtain
\begin{multline}
\label{eq:time-regularity-p_i^t}
|p_j^h(x) - p_i^h(x)| = |\exp(-u_j^h(x)) - \exp(-u_i^h(x))| \leq p_j^h(x) \vee p_i^h(x) \, | u_j^h(x) - u_i^h(x) | \\
\leq C (j - i) h  \exp(-c|x|^2) (1 + |x|^2)
\leq C(j - i)h \exp(-c|x|^2),
\end{multline}
which is our first claim.
This implies the \(\mathcal W_1\)--regularity of \(p_i^h\) as follows:
\begin{equation}
\label{eq:w1-time-regularity-p_i^t}
\mathcal W_1(p_j^h, p_i^h)
\leq \int |x| |p_j^h(x) - p_i^h(x)| dx \leq C (j - i) h \int |x| \exp(-c|x|^2) 
\leq C(j - i)h.
\end{equation}

\noindent{\rm (iii).} \emph{Uniform continuity of \(\frac{\delta F}{\delta p}\).}
Thanks to the estimate in \Cref{cor:p properties}, \(\{p_i^h\}_{h,i\leq N_h}\) forms a relatively compact set in \(\mathcal W_1\), and the \(\mathcal W_1\)--continuity of \(p \mapsto \frac{\delta F}{\delta p}(p,0)\) becomes uniform. That is, there exists a m.o.c. \(\varpi_0 : \mathbb R_+ \to \mathbb R_+\) such that
\[
\left| \frac{\delta F}{\delta p} (p_i^h, 0) - \frac{\delta F}{\delta p} (p_j^{h}, 0) \right| \leq \varpi_0 ( \mathcal W_1 (p_i^h, p_j^{h})), \quad \forall h > 0, \forall i \leq N_h, \forall j \leq N_{h}.
\]
Integrating along the straight line from \(0\) to any \(x \in \mathbb R^d\) and using the assumptions on \(\nabla \frac{\delta F}{\delta p}\), we obtain
\begin{align*}
\left| \frac{\delta F}{\delta p} (p_i^h, x) - \frac{\delta F}{\delta p} (p_j^{h}, x) \right| 
& \leq \left| \frac{\delta F}{\delta p} (p_i^h, 0) - \frac{\delta F}{\delta p} (p_j^{h}, 0) \right| \\
&\qquad + \int_0^1 \left | x \cdot \left( \nabla \frac{\delta F}{\delta p} (p_i^h, tx) - \nabla\frac{\delta F}{\delta p} (p_j^{h}, tx) \right) \right| dt \\
& \leq
\varpi_0 ( \mathcal W_1 (p_i^h, p_j^{h}))
+ L_G |x| \mathcal W_1 (p_i^h, p_j^{h}).
\end{align*}
Combining with \cref{eq:w1-time-regularity-p_i^t},we deduce that the exists a m.o.c.  $\varpi_1: \mathbb R_+ \to \mathbb R_+$ such that
\begin{equation}
\label{eq:continuity-deltaF-deltap}
\left| \frac{\delta F}{\delta p} (p_i^h, x) - \frac{\delta F}{\delta p} (p_j^{h}, x) \right| 
\leq (1+|x|) \varpi_1((j-i) h).
\end{equation}

\noindent{\rm (iv).} \emph{Time regularity of \(\lambda_i^h\).}
We first note that thanks to \cref{eq:sandwich-lambda_k^h} we can approximate \(\lambda_i^h\) by \(\int p_i^h B_i^h\), up to a uniform \(O(h)\) error. More precisely,
\begin{equation*}
|r_i^h| := \left| \lambda_i^h - \int p_i^h B_i^h \right|
\leq \left| \int (p_i^h - p_{i-1}^h ) B_i^h \right| \\
\leq C h \int \exp(-c|x|^2)(1+|x|^2) \leq Ch,
\end{equation*}
where we used \cref{eq:bound-B_k^h} and \cref{eq:time-regularity-p_i^t}.
It suffices then to study the difference
\[
\int p_j^h B_j^h - p_i^h B_i^h = \int p_i^h (B_j^h - B_i^h) + \int (p_j^h - p_i^h) B_j^h =: \delta + \delta'.
\]
We bound the second part, using again \cref{eq:bound-B_k^h}  and  \cref{eq:time-regularity-p_i^t},
\[
|\delta'| 
\leq C (j-i) h \int \exp(-c|x|^2) (1+|x|^2) \leq C(j-i)h.
\]
As for the first part, we decompose it into
three terms, each of which we treat separately:
\begin{multline*}
\delta = \frac{\sigma^2}{2} \int p_i^h (\Delta u_j^h - \Delta u_i^h) - \frac{\sigma^2}{4} \int p_i^h (|\nabla u_j^h|^2 - |\nabla u_i^h|^2) \\
+ \int p_i^h \left( \frac{\delta F}{\delta p} (p_j^h, \cdot) - \frac{\delta F}{\delta p} (p_i^h, \cdot) \right) 
=: \delta_1 + \delta_2 + \delta_3.
\end{multline*}
We apply integration by parts to the first term, using the previous estimates on \(\nabla u_i^h, p_i^h\) and the time regularity result of \(\nabla u_i^h\) from \Cref{lem:time-regularity-u_i^h} below,
\begin{equation*}
|\delta_1| = \frac{\sigma^2}{2} \left| \int p_i^h \nabla u_i^h \cdot (\nabla u_j^h - \nabla u_i^h) \right|
\leq C \int p_i^h (1+|x|)^2 ((j-i)h)^{1/2} \leq C((j-i)h)^{1/2}.
\end{equation*}
The second term is treated in the same way:
\[
|\delta_2| \leq \frac{\sigma^2}{4} \int p_i^h (|\nabla u_j^h | + |\nabla u_i^h|)  |\nabla u_j^h - \nabla u_i^h| \leq
C((j-i)h)^\frac 12.
\]
Using \cref{eq:continuity-deltaF-deltap}, we can then bound
\[
|\delta_3| \leq \int p_i^h \left| \frac{\delta F}{\delta p} (p_j^h, \cdot) - \frac{\delta F}{\delta p} (p_i^h, \cdot) \right|
\leq  \int p_i^h (1 + |x|) \varpi_1((j-i)h)
\leq C \varpi_1 ( (j-i)h ).
\]
Collecting the bounds on \(r,\delta' ,\delta\), we derive finally that
\begin{equation*}
|\lambda_j^h - \lambda_i^h| \leq |\delta| + |\delta'| + |r_j^h| +|r_i^h|
 \leq C\left(2((j-i)h)^\frac 12 + \varpi_1((j-i)h) + (j-i)h + 2h\right).
\end{equation*}
\end{proof}

\begin{lem}
\label{lem:time-regularity-u_i^h}
There exists a constant \(C\) such that for all \(h \in (0,1), i < j \leq N_h\), we have
\[
|\nabla u_{j}^h(x) - \nabla u^h_i(x)| \leq C
((j-i)h)^\frac 12(1+|x|),
\quad\forall x \in \mathbb R^d.
\]
\end{lem}

\begin{proof}
By taking spatial derivatives of the HJB equation \eqref{eq:u^h_i},
we see the following is satisfied for
\begin{equation}
\label{eq:pde-nabla-u_i^h}
\frac 1h (\nabla u_{k}^h - \nabla u_{k-1}^h)
= \frac{\sigma^2}{2} \Delta \nabla u_{k}^h - \frac{\sigma^2}{2} \nabla^2 u_{k}^h\, \nabla u_{k}^h + \nabla \frac{\delta F}{\delta p} (p_{k}^h,\cdot) =: \frac{\sigma^2}{2} \Delta \nabla u_{k}^h + A_{k}^h,
\end{equation}
where by estimates in \Cref{lem: nabla^2 u^i_h bounded} and \Cref{lem:nabla u^h_i(0) bound} we know that
\[
\sup_{h, i\leq N_h} |A_i^h (x)| \leq C(1 + |x|),\quad \forall x \in \mathbb R^d.
\]
The solution to \cref{eq:pde-nabla-u_i^h} admits the following representation
\[
\nabla u_{k}^h = \int_0^\infty e^{-h^{-1}t} (P_{\sigma^2 t} A_{k}^h + \frac 1h P_{\sigma^2 t} \nabla u_{k-1}^h) dt,
\]
where \(P_t\) is the heat kernel generated by \(\frac {1}2 \Delta\). Iterating this procedure with descending \(k\), we obtain
\begin{multline*}
\nabla u_j^h = \sum_{n=1}^{j-i} h^{-(n-1)} \int_{t_1,\ldots,t_n \geq 0} e^{-h^{-1}(t_1 + \ldots + t_n)} P_{\sigma^2 (t_1 +\cdots+ t_n)} A_{j+1-n}^h dt_1 \cdots dt_n \\
+ h^{-(j-i)} \int_{t_1,\ldots,t_{j-i} \geq 0} e^{-h^{-1}(t_1 + \ldots + t_{j-i})} P_{\sigma^2 (t_1+\cdots+t_{j-i})} \nabla u_i^h dt_1\cdots dt_{j-i}.
\end{multline*}
Here we used the semigroup property of the heat kernel. Denoting \(\gamma_{n,\theta}(t) = \frac{t^{n-1} e^{- \frac{t}{\theta}}}{\Gamma(n)\theta^n}\) the gamma distribution density, we have equivalently
\[
\nabla u_j^h = h \sum_{n=1}^{j-i} \int_0^\infty \gamma_{n, h}(t) P_{\sigma^2 t} A_{j+1-n}^h dt + \int_0^\infty \gamma_{j-i, h}(t) P_{\sigma^2 t} \nabla u_i^h dt.
\]
Subtracting \(\nabla u_i^h\), we obtain
\begin{multline*}
|\nabla u_j^h (x) - \nabla u_i^h (x) | \\
\begin{aligned}
& \leq h \sum_{n=1}^{j-i} \int_0^\infty \gamma_{n, h}(t) \left| P_{\sigma^2 t} A_{j+1-n}^h (x) \right| dt
+ \int_0^\infty \gamma_{j-i, h}(t) \left|P_{\sigma^2 t} \nabla u_i^h (x)- \nabla u_i^h (x)\right| dt \\
& \leq h \sum_{n=1}^{j-i} \int_0^\infty \gamma_{n, h}(t) C (1 + |x| + (\sigma^2 t)^{1/2}) dt+ \int_0^\infty \gamma_{j-i, h}(t) (\sigma^2 t)^{1/2} || \nabla^2 u_i^h ||_\infty dt \\
& \leq  C(j-i) h (1 + |x|) + Ch^{3/2} \sum_{n=1}^{j-i} \frac{\Gamma(n+\frac 12)}{\Gamma(n)} + C h^{1/2} \frac{\Gamma(j-i+ \frac 12)}{\Gamma(j-i)} \\
& \leq C(j-i) h (1+|x|) + C ((j-i)h)^{3/2} + C ((j-i)h)^{1/2}.
\end{aligned}
\end{multline*}
In the second inequality, we used the following properties of the heat kernel: \(P_t |\cdot| (x) \leq c_d \sqrt{t} + |x|\), \(||P_t f - f||_\infty \leq \sqrt t || f||_\text{Lip}\).
In the last inequality, we used the log-convexity of the gamma function along the positive real line: \(\Gamma(x + \frac 12) \leq \sqrt{\Gamma(x) \Gamma(x+1)} = \sqrt x \Gamma(x)\) for \(x > 0\).
\end{proof}

\subsection{Proof of \Cref{thm:gradientflow}}
\label{sec:gradientflow-proof}

\begin{proof}[Proof of \Cref{thm:gradientflow}]
\noindent{\rm (i).}\quad 
Let us define by abuse of notations the step flows
\begin{equation*}
 f^h(t) =  f^h_i, \quad \text{for } t \in[ih,(i+1)h), \quad f = p, \lambda.
\end{equation*}
In view of \Cref{cor:p properties} and \Cref{lem:compactness-lambda^h_i}, we can apply a version of Arzel\`a--Ascoli Theorem for discontinuous functions,  see \textit{e.g.}\cite[Theorem 6.1]{droniou16}, to ensure that the family of functions $({p}^h)_h$ (resp. $(\lambda^h)_h$) is relatively compact in $B([0,T]\times\dbR^d)$ (resp.  $B([0,T])$) the space of bounded functions on $[0,T]\times\dbR^d$ (resp.  $[0,T]$)  equipped with the uniform norm, and	any adherence values $p$ (resp. $\lambda$) is uniformly continuous. 

Let $p$ and $\lambda$ be such adherence values, \ie, there exists $h_n\downarrow 0$ such that \( p^{h_n} \to p\) and \( \lambda^{h_n} \to \lambda\) uniformly. 
Note that $\psi^{h_n}:=\sqrt{p^{h_n}}$ also converges to $\psi := \sqrt{p}$ uniformly on $[0,T]\times\dbR^d$ by using the elementary inequality \( |\sqrt a - \sqrt b| \leq \sqrt{|a-b|}\). 
 
 \noindent{\rm (ii).}\quad Let us verify that the limit \((p,\psi,\lambda)\) solves the MFS equation~\eqref{eq:mfSchr\"odinger-root} in the weak sense,
\ie, for all \(\varphi \in C_c^2(\mathbb R^d)\), we have for all \(t \in [0,T]\),
\begin{multline}
\label{eq:weak-solution-mf-Schr\"odinger}
\int (\psi(t,x) - \psi(0,x)) \varphi(x) dx \\
= \int_0^t \int \frac{\sigma^2}{2} \psi(s,x) \Delta \varphi(x)
- \frac 12\Big( \frac{\delta F}{\delta p} (p_s, x)
- \lambda(s)\Big) \psi(s,x) \varphi(x) dx ds.
\end{multline}
By construction, we know that the following holds for \(i \leq N_h\),
\begin{multline}
\label{eq:discrete-weak-solution-mf-Schr\"odinger}
\int \sum_{k=1}^{i} \log \frac{\psi^h(kh,x)}{\psi^h((k-1)h,x)} \psi^h(kh,x)\varphi(x) dx \\
= h \sum_{k=1}^i \int \frac{\sigma^2}{2} \psi^h(kh,x) \Delta \varphi(x)
- \frac 12 \Big( \frac{\delta F}{\delta p} (p^h_{kh}, x) 
- \lambda^h(kh)\Big) \psi^h(kh,x) \varphi(x) dx.
\end{multline}
Let \(i = \lfloor t/h \rfloor\) be the unique integer such that \(t \in [ih, (i+1)h)\) and
denote the difference between the left and right hand sides of \cref{eq:weak-solution-mf-Schr\"odinger,eq:discrete-weak-solution-mf-Schr\"odinger} by \(\delta^\ell(h), \delta^r(h)\) respectively.
We want to show that both \(\delta^\ell(h_n), \delta^r(h_n)\) converge to zero when \(n \to \infty\), so that \cref{eq:weak-solution-mf-Schr\"odinger} is proved.
For the left hand side we have
\(\delta^\ell(h) = \delta^\ell_1(h) + \delta^\ell_2(h)\)
with
\begin{align*}
\delta^\ell_1(h) &= \int (\psi (t,x) - \psi^h(t,x)) \varphi(x) dx, \\
\delta^\ell_2(h) &=
\int \sum_{k=1}^i \left( \psi^h(kh,x) - \psi^h((k-1)h,x) - \log \frac{\psi^h(kh,x)}{\psi^h((k-1)h,x)} \psi^h(kh,x) \right) \varphi(x) dx.
\end{align*}
The first part converges to \(0\) along the sequence \(h_n\) as \(\psi^{h_n} \to \psi\) uniformly. 
For the second part we note that, by using \cref{eq:uih_time_regularity},
\begin{multline*}
\left|\psi^h(kh,x) - \psi^h((k-1)h,x) - \log \frac{\psi^h(kh,x)}{\psi^h((k-1)h,x)} \psi^h(kh,x)\right| \\
\begin{aligned}
 & = \left|\exp(-u^h_k(x)/2) - \exp(-u^h_{k-1}(x)/2) + \frac 12 \exp(-u^h_k(x)/2) (u^h_k(x) - u^h_{k-1}(x))\right| \\
 & \leq \frac18 \max\{\psi^h_k(x),\psi^h_{k-1}(x)\}| u^h_k(x) - u^h_{k-1}(x)|^2
\leq C \exp(-c|x|^2)h^2,
\end{aligned}
\end{multline*}
so that \(\delta^\ell_2(h) \leq C h \int \exp(-c|x|^2) \varphi(x) dx \leq Ch\).
For the right hand side, we have \(\delta^{r}(h) = \delta^r_1(h) + \delta^r_2(h)\) with
\begin{align*}
\delta^r_1(h) &=
\int_{ih}^t \int \frac{\sigma^2}{2} \psi(s,x) \Delta \varphi(x)
-  \frac 12\Big( \frac{\delta F}{\delta p} (p_s, x)
- \lambda(s)\Big) \psi(s,x) \varphi(x) dx ds, \\
\delta^r_2(h) &=
\begin{aligned}[t]
&\int \int_0^{ih} \frac{\sigma^2}{2} (\psi - \psi^h) (s,x) \Delta \varphi(x) \\
&- \frac 12 \left( \frac{\delta F}{\delta p} (p_\cdot,\cdot) \psi - \frac{\delta F}{\delta p} ( p^h_\cdot,\cdot) \psi^h - \lambda \psi + \lambda^h\psi^h\right)(s,x) \varphi(x) dx ds,
\end{aligned} 
\end{align*}
The first part clearly satisfies \(|\delta^r_1(h)| \leq Ch\) while 
the second part goes to zero along the sequence \(h_n\)  as \((p^{h_n},\psi^{h_n},\lambda^{h_n}) \to (p,\psi,\lambda)\) uniformly.

\noindent{\rm (iii).}\quad 
If we denote \(c(t,x) :=  \frac{\delta F}{\delta p}(p_t,x) - \lambda(t)\), then Step~(ii) ensures that $\psi$ is a weak solution to the linear PDE
\begin{equation*}
\partial_t \psi_t = \frac{\si^2}{2} \Delta \psi_t - \frac12 c_t \psi_t.
\end{equation*}
By weak uniqueness and strong existence, it is actually the classical solution to this PDE.  It follows that $p_t=\psi_t^2$  satisfies \cref{eq:mfSchr\"odinger}  with $\lambda_t=\lambda(p_t)$ as the mass of $p_t$ is conserved to $1$ by construction. We conclude by uniqueness stated in \Cref{thm:wellposedness}.

\end{proof}

\appendix
\section{Appendix}\label{sec:appendix}

\subsection{Regularity of Solution to HJB Equation}\label{sec:regularity}

Throughout this section, we assume that Assumptions~\ref{assum:decomposition-F-derivative}, \ref{assum:initialdistribution} and \ref{assum:mLipschitz} hold and we fix a time horizon $T<+\infty.$
Let $u$ be the unique viscosity solution to the HJB equation \eqref{eq:fixpoint_m-to-u}.  We start by establishing upper and lower bounds on $u.$ 

\begin{lem}\label{lem:bound of u}
It holds for all $t\in[0,T],$ $x\in \dbR^d,$
\[-C_T \le u(t,x) \le C_T(1+|x|^2).\]
\end{lem}

\begin{proof}
Under \Cref{assum:decomposition-F-derivative} and \ref{assum:mLipschitz} we have $-C_T \le \frac{\d F}{\d p}(m_t, x) \le C_T (1+|x|^2)$. Additionally, under \Cref{assum:initialdistribution}, the initial value satisfies $-C\le u_0(x)\le C(1+|x|^2)$. The desired result follows from the comparison principle.
\end{proof}

To show existence and uniqueness of the classical solutions to HJB equation~\eqref{eq:fixpoint_m-to-u}, it is convenient to consider the change of variable $\psi := e^{-\frac12 u},$ which corresponds to the well-known Cole--Hopf transformation.

\begin{lem}\label{lem:feynmankac psi}
The function  $\psi$ is the unique viscosity solution to
\begin{equation}\label{eq:Schr\"odinger}
 \partial_t \psi_t = \frac{\si^2}{2} \Delta \psi_t - \frac12 \left(\frac{\d F}{\d p}(m_t, \cdot) -\gamma u_t \right)\psi_t, \quad \psi_0(x) = e^{-\frac12 u_0(x)}.
\end{equation}
Moreover, it admits the following probabilistic representation
\begin{equation}\label{eq:FeynmanKacLemma u}
 \psi(t,x) = \dbE\left[ e^{-\frac12 \int_0^t\left(\frac{\d F}{\d p}(m_{t-s}, x+\si W_s) - \gamma u(t-s, x+\si W_s)\right)ds} \psi_0(x+\si W_t)\right].
\end{equation}
\end{lem}

\begin{proof}
First, it follows from the monotonicity of  $x\mapsto e^{-\frac12 x}$ that $\psi$ is a viscosity solution to \cref{eq:Schr\"odinger} if and only if $u$ is a viscosity solution to \cref{eq:fixpoint_m-to-u}. 
Then, by the bound of $u$ in \Cref{lem:bound of u}, we have
\[\dbE\left[ e^{\frac{\gamma}{2}\int_0^t u(t-s, x+\si W_s) ds} \right]
\le \frac{1}{t} \int_0^t \dbE\left[ e^{\frac{\gamma t}{2} u(t-s, x+\si W_s) }\right] ds \le \frac{1}{t} \int_0^t \dbE\left[ e^{\frac{\gamma t C_T}{2}(1+ |x+\si W_s |^2) }\right] ds <\infty, \]
for all $t\le \d$ with $\d$ small enough. Also note that $\big(-\frac{\d F}{\d p}(m_t,\cdot)\big)_{t\in[0,T]}$ and $\psi_0$ are bounded from above. So for $t\le \d$ we may define
\[\tilde \psi(t,x) := \dbE\left[ e^{-\frac12 \int_0^t\left(\frac{\d F}{\d p}(m_{t-s}, x+\si W_s) - \gamma u(t-s, x+\si W_s)\right)ds} \psi_0(x+\si W_t)\right]. \]
It is easy to verify that $\tilde\psi$ is a viscosity solution to \cref{eq:Schr\"odinger}, so equal to $\psi$ on $[0, \d]$. Also note that $\psi = e^{-\frac12 u} \le C_T$ thanks to \Cref{lem:bound of u}. So we may further define for $t\in (\d, 2\d],$
\begin{align*}
 \tilde \psi(t,x) &:= \dbE\left[ e^{-\frac12 \int_0^{t-\d}\left(\frac{\d F}{\d p}(m_{t-s}, x+\si W_s) - \gamma u(t-s, x+\si W_s)\right)ds} \tilde \psi(\d, x+\si W_{t-\d})\right]\\
 &= \dbE\left[ e^{-\frac12 \int_0^t\left(\frac{\d F}{\d p}(m_{t-s}, x+\si W_s) - \gamma u(t-s, x+\si W_s)\right)ds} \psi_0(x+\si W_t)\right].
\end{align*}
Therefore the desired probabilistic representation \eqref{eq:FeynmanKacLemma u} follows from induction.
\end{proof}

\begin{prop}
\label{prop:regularity-heat-equation}
 The function $\psi=e^{-\frac{u}{2}}\in C^{3}(Q_T)\cap C(\bar Q_T)$ is the unique classical solution to \cref{eq:Schr\"odinger}.
 Moreover,  the gradient $\nabla \psi$ satisfies the growth condition $|\nabla \psi(t,x)| \le C_T(1+|x|^2)$.
\end{prop}

\begin{proof}
It follows from \Cref{lem:feynmankac psi} that
 \[\left(e^{-\frac12 \int_0^{s}\left(\frac{\d F}{\d p}(m_{t-r}, x+\si W_r)-\gamma u(t-r, x+\si W_r) \right) dr}\psi(t-s, x+\si W_s)\right)_{s\in [0,t]}\] 
 is a continuous martingale. By 
It\^o's formula, we have for all $0\le r\le t$ that
\begin{multline}\label{eq:dumel}
 \psi (t,x) = \dbE\Big[\psi(r, x+\si W_{t-r})  \\
  - \frac12\int_0^{t-r} \Big( \frac{\d F}{\d p}(m_{t-s}, x+\si W_s) -\gamma u(t-s, x+\si W_s) \Big)\psi (t-s,x+\si W_s) ds\Big].
\end{multline}
Recall that $|\frac{\d F}{\d p}(m_t, x)| + |u(t,x)| \le C_T(1+|x|^2)$ on $[0,T]\times\dbR^d$, so for all $t\le T$ we have
\[ \int_0^t \dbE\left[\Big| \Big( \frac{\d F}{\d p}(m_{t-s}, x+\si W_s) - \gamma u(t-s, x+\si W_s)\Big) \psi (t-s,x+\si W_s) \frac{W_s}{\sigma s} \Big|\right] ds <\infty. \]
As a result $\nabla \psi$ exists and is equal to
\begin{multline*}
 \nabla \psi (t,x) = \dbE\Big[ \nabla \psi_0(x+ \si W_t)\\
  - \frac12\int_0^t \Big(\frac{\d F}{\d p}(m_{t-s},x+ \si W_s ) - \gamma u(t-s, x+\si W_s)\Big)\psi (t-s,x+ \si W_s) \frac{W_s}{\si s} ds  \Big].
\end{multline*}
Therefore we obtain $|\nabla \psi(t,x)| \le C_T(1+|x|^2) $, and
\[|\nabla u (t,x)|= 2\left|\frac{\nabla \psi (t,x)}{ \psi(t,x)}\right|\le C_T(1+|x|^2)e^{C_T(1+|x|^2)}.\]
In particular we have $\dbE\big[|\nabla u(t, x+\si W_s)|^2\big]<\infty$ for $s$ small enough. So for $r<t$ and $r$ close enough to $t$ we have
\begin{multline*}
 \nabla \psi (t,x) = \dbE\Big[ \nabla \psi(r, x+ \si W_{t-r}) \\
 -\frac12\int_0^{t-r} \Big(\nabla\frac{\d F}{\d p}(m_{t-s},x+ \si W_s)
 - \gamma \nabla u (t-s,x+ \si W_s) \Big) \psi (t-s,x+ \si W_s)ds   \Big].
\end{multline*}
Further note that
\[\int_0^{t-r} \dbE\left[\Big| \Big(\nabla\frac{\d F}{\d p}(m_{t-s},x+ \si W_s)
 - \gamma \nabla u (t-s,x+ \si W_s) \Big) \psi (t-s,x+ \si W_s) \frac{W_s}{\si s} \Big|\right] ds <\infty, \]
So $\nabla^2  \psi $ exist and is equal to
\begin{multline*}
 \nabla^2 \psi (t,x) = \dbE\Big[ \nabla\psi(r, x+ \si W_{t-r}) \frac{W_{t-r}}{\si (t-r)} \\
 - \frac12\int_0^{t-r} \Big(\nabla\frac{\d F}{\d p}(m_{t-s},x+ \si W_s)
 - \gamma \nabla u (t-s,x+ \si W_s) \Big) \psi (t-s,x+ \si W_s) \frac{W_s}{\si s} ds \Big].
\end{multline*}
Further, in order to compute the time partial derivative, recall \cref{eq:dumel}.
Since we have already proved that $x\mapsto \psi(t, x)$ belongs to $C^2$, it follows from It\^o's formula that
\begin{multline*}
 \psi (t,x) - \psi(r, x)= \dbE\Big[ \int_{0}^{t-r} \Big(\frac{\si^2}{2} \Delta \psi(r,x+ \si W_s)\\
-\frac12\Big(\frac{\d F}{\d p}(m_{t-s}, x+ \si W_s) -\gamma u(t-s, x + \si W_s)\Big)
 \psi (t-s,x+ \si W_s)\Big) ds \Big].
\end{multline*}
Then clearly $\partial_t \psi$ exists and $\psi$ satisfies \cref{eq:Schr\"odinger}. in the classical sense. Moreover, using the same argument, we can easily show that $\nabla^3 \psi$ and $\partial_t \nabla \psi$ exist and are continuous on $Q_T$.
\end{proof}

\subsection{Gaussian Bounds}
\label{sec:gaussian}

The aim of this section is to establish a technical result which ensures that if a family of probability distributions writes as the exponential of a sum of a Lipschitz and a convex function then it admits uniform Gaussian bounds.

\begin{lem}\label{lem:property-p-app}
Let $\cT$ be an index set. We assume that the family of probability measures $p_t = e^{-(v_t+w_t)},\, t\in\cT$, satisfies the following conditions:
\begin{itemize}
 \item[\rm (i)] For some $\ol\eta>\ul\eta>0,$ it holds $\ul\eta I_d \le \nabla^2 v_t \le \ol\eta I_d$ for all $t\in \cT$.
 \item[\rm (ii)] The vectors $\nabla v_t(0)$ are uniformly bounded, \ie, $\sup_{t\in \cT} |\nabla v_t(0)| <\infty $.
 \item[\rm (iii)] The gradients $\nabla w_t$ are uniformly bounded, \ie, $\sup_{t\in \cT} \|\nabla w_t\|_\infty <\infty $.
\end{itemize}
Then there exist $\ul c, \ol c, \ul C, \ol C>0,$ such that for all $t\in\cT,$ $x\in\dbR^d,$
\begin{equation*}
  \ul C e^{- \ul c |x|^2} \le  p_t(x) \le \ol C e^{- \ol c |x|^2} 
\end{equation*}
\end{lem}

\begin{proof}
We decompose the probability measure \(p_t = q_tr_t\) with \(q_t = e^{- v_t} / \int e^{- v_t} \) and \(r_t = e^{- w_t} / \int e^{- w_t} q_t\).

\noindent{\rm (i).}\quad We first derive some estimates on \(v\) and the corresponding measure \(q\). From Assumption~(i), the following inequalities holds
\[
 \left| \nabla v_t (x) - \nabla v_t (0) \right| \left| x \right|\geq \left( \nabla v_t (x) - \nabla v_t(0) \right) \cdot x \geq \ul\eta \left| x\right|^2
\]
For each \(t \in \cT \), let \(x_t\) be the unique solution to \(\nabla v_t \left(x\right) = 0\), \ie, $x_t$ is the minimizer of \(v\). Plugging \(x_t\) in the inequality above, we obtain
\(\left|\nabla v_t \left( 0 \right) \right| \left| x_t\right| \geq c \left| x_t\right|^2\). Thus, in view of Assumption~(ii),  \(x_t\) is bounded in \(\mathbb R^d\), \ie, 
\begin{equation}\label{eq:x_t}
 \sup_{t \in \cT} \left|x_t\right| < +\infty.
\end{equation}
Denote \(\tilde v_t \left(x\right) = v_t \left(x\right) - v_t \left(x_t\right)\).We have by definition \(q_t = e^{- \tilde{v}_t} / \int e^{- \tilde{v}_t}\) and \(\tilde v_t \left(x_t\right) = 0\) as well as \(\nabla \tilde v_t \left(x_t\right) = 0\).  It follows from Taylor expansion that
\begin{equation*}
 \frac 12 \ul\eta \left| x - x_t\right|^2 \leq \tilde v_t \left(x\right) \leq \frac 12 \ol\eta \left| x - x_t\right|^2, 
\end{equation*}
so that 
\begin{equation}\label{eq:q_t}
\left( \frac{\ul\eta }{2 \pi} \right)^{d/2} \exp \left( - \frac{\ol\eta}{2} \left| x - x_t \right|^2 \right) \leq q_t \leq \left( \frac{\ol\eta}{2 \pi} \right)^{d/2} \exp \left( - \frac{\ul\eta}{2} \left| x - x_t \right|^2 \right).
\end{equation}

\noindent{\rm (ii).}\quad Now we estimate the function \(r\). Denote \(\tilde w_t \left(x\right) = w_t \left(x\right) - w_t \left(x_t\right)\). We have by definition \(r_t = e^{- \tilde w_t} / \int e^{- \tilde w_t} q_t\) and \(\tilde w_t \left(x_t\right) = 0\). Thanks to Assumption~(iii), we know that \(\nabla w_t = \nabla \tilde w_t\) is uniformly bounded by some constant, denoted \(L\). Therefore it holds
\begin{equation*}
 -L \left| x - x_t\right| \leq \tilde{w}_t \left(x \right)  \leq L \left| x - x_t\right|.
\end{equation*}
In particular, in view of \cref{eq:q_t} and \cref{eq:x_t}, it holds for some $\ul C, \ol C>0,$
\begin{equation}\label{eq:r_t}
 \ul C \exp \left( -L \left| x - x_t\right| \right) \leq r_t  \leq \ol C \exp \left( L \left| x - x_t\right| \right).
\end{equation}
\noindent{\rm (iii).}\quad Since $p_t=q_t r_t,$ the conclusion follows immediately from \cref{eq:x_t}, \cref{eq:q_t} and \cref{eq:r_t}.
\end{proof}

\subsection{Reflection Coupling}\label{sec:reflectioncoupling}

In the section we recall the reflection coupling technique developped in \cite{eberle11, eberle2019quantitative}
and use it to estimate the $\cW_1$--distance between the marginal laws of two diffusion processes with drift $b$ and $b+\delta b.$

\begin{assum}\label{assum:reflectioncoupling}
The drifts $b$ and $\delta b$ satisfy
\begin{itemize}
 \item[(i)] $b$ and $\delta b$ are Lipschitz in $x$,  \ie, there is a constant $L> 0$ such that
 \[|b(t,x) -b(t,y)| + |\delta b(t,x) - \delta b(t,y)|  \le L|x-y|, \quad\mbox{for all $t\in [0,T],$ $x,y\in\dbR^d$}; \]

 \item[(ii)] there exists a continuous function $\kappa:(0,\infty)\rightarrow \dbR$ such that $\limsup_{r\rightarrow\infty} \kappa(r) <0,$ $\int_0^1 r \kappa^+(r) dr<\infty$ and
 \[(x-y)\cdot \big(b(t,x) -b(t,y)\big) \le \kappa(|x-y|) |x-y|^2, \quad\mbox{for all $t\in [0,T],$ $ x,y\in\dbR^d.$}\]

\end{itemize}
\end{assum}

\begin{rem}\label{rem:decom-drift}
If $b(t,x) = -(\alpha(t,x) + \nabla \beta(t,x))$ with $\alpha$ bounded and $\beta$ $\eta$--convex in $x$, \ie,
\begin{equation*}
 \left(\nabla \beta(t,x) - \nabla \beta(t,y)\right)\cdot \left(x-y\right) \geq \eta \left|x-y\right|^2,
\end{equation*}
then the function $b$ satisfies \Cref{assum:reflectioncoupling} (ii) with $\kappa(r)=\frac{2 \|\alpha\|_{\infty}}{r} - \eta.$
\end{rem}

\begin{thm}\label{thm:reflectioncoupling}
Let \Cref{assum:reflectioncoupling} hold. Consider the following two diffusion processes
\begin{equation*}
 dX_t = b(t, X_t) dt + \si dW_t,\qquad dY_t = (b + \d b)(t, Y_t) dt +\si dW_t,
\end{equation*}
and denote their marginal distributions by $p^X_t :=\cL(X_t)$ and $p^Y_t := \cL(Y_t)$. Then we have
\begin{equation}\label{eq:coupling_contraction}
 \cW_1(p^X_t, p^Y_t) \le C e^{- c \si^2 t} \Big(\cW_1(p^X_0, p^Y_0)+\int_0^t e^{c \si^2 s}\dbE\big[|\d b(s,Y_s)|\big] ds \Big) , \quad\mbox{for all $t\ge 0$},
\end{equation}
where the constants $C$ and $c$ only depend on the function $\kappa(\cdot)/\si^2$.
\end{thm}

\begin{rem}\label{rem:ergodicity}
It follows immediately from \Cref{thm:reflectioncoupling} that if $\pi$ is an invariant distribution of the process $X$ then
\begin{equation*}
 \cW_1(p^X_t, \pi) \le C e^{- c \si^2 t} \cW_1(p^X_0,\pi), \quad\mbox{for all $t\ge 0$}.
\end{equation*}
In particular, it is unique and it is the limiting distribution of $X.$
\end{rem}

\begin{proof}
We first recall the reflection-synchronuous coupling introduced in \cite{eberle2019quantitative}. Introduce Lipschitz functions ${\rm rc}:\dbR^d\times\dbR^d\mapsto [0,1]$ and ${\rm sc}:\dbR^d\times\dbR^d\mapsto [0,1]$ satisfying
\[{\rm rc}^2(x,y) + {\rm sc}^2 (x,y) =1.\]
Fix a small constant $\eta>0$. We impose that $\rc(x,y)=1$ whenever $|x-y|>\eta$ and $\rc(x,y)=0$ if $|x-y|\le \eta/2$.
The so-called reflection-synchronuous coupling is the strong solution to the following SDE system:
\begin{align*}
 dX_t &= b(t, X_t) dt + \rc(X_t, Y_t)\si dW^1_t + \syc(X_t, Y_t) \si dW^2_t, \\
 dY_t &= (b + \d b)(t, Y_t) dt + \rc(X_t, Y_t) (I - 2 e_t \langle e_t, \cdot \rangle )\si dW^1_t + \syc(X_t, Y_t)\si dW^2_t ,
\end{align*}
where $W^1, W^2$ are $d$-dimensional independent standard Brownian motion and
\[e_t = \frac{X_t-Y_t}{|X_t-Y_t|}\quad\mbox{for $X_t \neq Y_t$, \quad and}\quad e_t = u \quad\mbox{for $X_t = Y_t,$}\]
with $u\in \dbR^d$ a fixed arbitrary unit vector. We denote by $\rc_t := \rc(X_t, Y_t)$ and define $r_t : = |X_t-Y_t|$. Observe that
\[dr_t = \langle e_t, b(t,X_t)- b(t, Y_t) -\d b(t, Y_t)\rangle dt +2 \rc_t \si dW^\circ_t,\]
where $W^\circ$ is a one-dimensional standard Brownian motion, see \cite[Lemma 6.2]{eberle11}.

Next we construct an important auxiliary function $f$ as in \cite[Section 5.3]{eberle2019quantitative}. First define two constants:
\begin{align*}
 R_1 & = \inf\{R\ge 0:~ \kappa(r)\le 0,~\mbox{for all $r\ge R$} \},\\
 R_2 & = \inf\{R\ge R_1 : ~ \kappa(r) R(R-R_1)\le -4\si^2 , ~\mbox{for all $r\ge R$}\}.
\end{align*}
Further define
\[\f(r) = e^{-\frac{1}{2 \sigma^2}\int_0^r u\kappa^+(u)du}, \quad \Phi(r) = \int_0^r \f(u)du, \quad g(r) = 1-\frac{c}{2}\int_0^r \Phi(u) \f(u)^{-1}du,\]
where the constant $c=\left(\int_0^{R_2} \Phi(r)\f(r)^{-1} dr\right)^{-1}$, and eventually define the auxiliary function
\[f(r) = \int_0^r \f(u) g(u\wedge R_2)du.\]
One easily checks that 
\begin{equation*}
 r\f(R_1) \le \Phi(r)  \le 2f(r) \le 2\Phi(r) \le 2r, \quad\mbox{for all $r>0$}.
\end{equation*}
Note also that $f$ is increasing and concave. In addition, $f$ is linear on $[R_2,+\infty)$, twice continuously differentiable on $(0, R_2)$ and satisfies
\begin{equation}\label{eq:fsubODE}
 2\si^2f''(r) \le -r\kappa^+(r) f'(r) -c\si^2 f(r), \quad\mbox{for all $r\in (0,\infty) \backslash \{R_2\}.$}
\end{equation}
This inequality follows easily by direct computation on $[0,R_2)$ and we refer to \cite[Eqn (5.32)]{eberle2019quantitative} for a detailed justification on $(R_2,+\infty)$.
Then we have by It\^o--Tanaka formula as in \cite[Eqn~(5.26)]{eberle2019quantitative} that
\[d f(r_t) \le \left(f'_{-}(r_t)\langle e_t, b(t, X_t)-b(t, Y_t) - \d b(t, Y_t)\rangle +2 \sigma^2 \rc_t^2 f''(r_t) \right) dt +2 \rc_t f'_{-}(r_t)\si dW^\circ_t.\]
Further note that 
\[\langle e_t, b(t, X_t)-b(t, Y_t)\rangle
\le 1_{r_t<\eta} |b|_{Lip} \eta +1_{r_t\ge \eta} r_t \kappa^+(r_t).\]
Together with the fact that $f'\le 1,$ $f''\le 0$ and $\rc_t 1_{r_t\ge \eta}= 1,$ we deduce that
\begin{multline*}
 d e^{c\si^2 t}f(r_t)
\le e^{c\si^2 t} \Big( 2\rc_t f'_{-}(r_t)\si dW^\circ_t + |\d b(t,Y_t)| dt + 1_{r_t<\eta} (c\si^2 f(r_t) +|b|_{Lip} \eta)dt\\
 + 1_{r_t\ge \eta} \big(c\si^2 f(r_t) +r_t \kappa^+(r_t) f'(r_t) +2\si^2 f''(r_t) \big)dt \Big).
\end{multline*}
It follows from \cref{eq:fsubODE} that
\[d e^{c\si^2 t}f(r_t)
\le e^{c\si^2 t} \left( 2\rc_t f'_{-}(r_t)\si dW^\circ_t + \Big(|\d b(t,Y_t)|+ (c\si^2 +|b|_{Lip}) \eta\Big) dt\right).\]
Taking expectation on both sides, we obtain
\[ \dbE[e^{c\si^2 t} f(r_t) - f(r_0)] \le \int_0^t e^{c\si^2 s} \Big( \dbE \big[ |\d b(s,Y_s)| \big] + (c\si^2 +|b|_{Lip}) \eta\Big) ds. \]
Again due to the construction of $f$ we have
\begin{multline*}
 \cW_1(p^X_t, p^Y_t) \le \dbE[r_t] \le 2 \f(R_1)^{-1}\dbE[f(r_t)]
\\
\begin{aligned}
 & \le 2 \f(R_1)^{-1} e^{ - c\si^2 t}\left(\dbE[f(r_0)] +\int_0^t e^{c\si^2 s} \Big( \dbE \big[ |\d b(s,Y_s)| \big] + (c\si^2 +|b|_{Lip}) \eta\Big) ds \right)\\
 & \le 2  \f(R_1)^{-1} e^{ - c\si^2 t}\left(\cW_1(p^X_0, p^Y_0) +\int_0^t e^{c\si^2 s} \Big( \dbE \big[ |\d b(s,Y_s)| \big] + (c\si^2 +|b|_{Lip}) \eta\Big) ds \right).
\end{aligned}
\end{multline*}
By passing to the limit $\eta\rightarrow 0$, we finally obtain the estimate \eqref{eq:coupling_contraction}.
\end{proof}

\bibliography{references}
\bibliographystyle{plain}

\end{document}